\documentclass[12pt, reqno]{amsart} 
\usepackage{amssymb,amscd,amsfonts,amsbsy}
\usepackage{latexsym}
\usepackage{exscale}
\usepackage{amsmath,amsthm,amsfonts}
\usepackage{mathrsfs}
\usepackage{xcolor} 
\usepackage[colorlinks=true,linkcolor=blue,citecolor=blue,urlcolor=blue]{hyperref} 
\usepackage{esint} 
\usepackage{amssymb} 
\usepackage{stmaryrd}
\usepackage{cite} 
\usepackage{tikz}
\calclayout

\allowdisplaybreaks[3] 

\numberwithin{equation}{section}

\newtheorem{cor}{Corollary}[section]
\newtheorem{thm}{Theorem}[section]
\newtheorem{lem}{Lemma}[section]
\newtheorem{defn}{Definition}[section]
\newtheorem{rem}{Remark}[section]
\newtheorem{pro}{Proposition}[section]

\makeatletter
\@addtoreset{equation}{section}
\makeatother

\def\rn{{\mathbb R^n}}

\def\n{{\mathbb N}}

\def\BMO{{\rm BMO}}

\def\M{{\mathscr M}}
\def\I{{\mathscr I}}
\def\T{{\mathscr T}}

\def\A{{\mathscr A}}

\def\P{{\mathscr P}}

\def\F{{\mathscr F}}
\def\pp{{p(\cdot)}}

\def\Lpp{{L^\pp}}
\def\cpp{{p'(\cdot)}}
\def\Lcpp{L^\cpp}

\def\bbA{\mathbb{A}}

\DeclareMathOperator{\dist}{dist}

\newcommand{\subRn}{{{\mathbb R}^n}}
\newcommand{\Q}{{Q_{j}^{k}}}

\newcommand{\avf}{{\langle f\rangle}}

\newcommand{\pap}{{p_{1}(\cdot)}}
\newcommand{\pbp}{{p_{2}(\cdot)}}
\newcommand{\cpap}{{p_{1}^{\prime}(\cdot)}}
\newcommand{\cpbp}{{p_{2}^{\prime}(\cdot)}}

\newcommand{\Lp}{L^{p(\cdot)}}

\newcommand{\qq}{{q(\cdot)}}
\newcommand{\rr}{{r(\cdot)}}
\newcommand{\sst}{{s(\cdot)}}

\newcommand{\Lq}{L^{q(\cdot)}}

\def\vf{\mathbf{f}}
\def\vg{\mathbf{g}}

\newcommand{\vv}{\mathbf{v}}

\newcommand{\ve}{\mathbf{e}}

\def\F{\mathcal{F}}

\def\Z{\mathbb{Z}}

\def\R{\mathbb{R}}

\def\dist{\operatorname{dist}}

\def\BMO{\operatorname{BMO}}

\newcommand{\D}{\mathscr{D}}

\newcommand{\calM}{\mathcal{M}}
\newcommand{\calS}{\mathcal{S}}

\begin{document}

\title[New fractional type weights...]
{\bf New fractional type weights and the boundedness of some operators}

\author[X. Cen]{Xi Cen$^{*}$}
\address{Xi Cen\\
Department of applied mathmatics\\
	School of Mathematics and Physics\\
	Southwest University of Science and Technology\\
	Mianyang 621010 \\
	People's Republic of China}\email{xicenmath@gmail.com}

\author[Q. He]{Qianjun He}
\address{Qunjun He\\
	School of Applied Science\\
	Beijing Information Science and Technology University\\
	Beijing 100192\\
	People's Republic of China}\email{qjhe@bistu.edu.cn}

\author[Z. Song]{Zichen Song}
\address{Zichen Song\\
	School of Mathematics and Stastics\\
	Xinyang Normal University\\
	Xinyang 464000\\
	People's Republic of China}\email{zcsong@aliyun.com}

\author[Z. Wang]{Zihan Wang}
\address{Zihan Wang\\
School of Statistics and Mathematics\\
Shandong University of Finance and Economics\\
Jinan {{250014}}\\
People's Republic of China}\email{zihwang@aliyun.com}


\subjclass[2020]{42B25, 42B20, 42B35.}

\keywords{Multilinear; Fractional maximal operators; Fractional integral operators; Commutators; Variable exponents Lebesgue spaces; Multiple weights; Matrix weights; Rubio de Francia extrapolation.}

\thanks{$^{*}$ Corresponding author, Email: xicenmath@gmail.com}

\begin{abstract} 
Two classes of fractional type variable weights are established in this paper. The first kind of weights ${A_{\vec p( \cdot ),q( \cdot )}}$ are variable multiple weights, which are characterized by the weighted variable boundedness of multilinear fractional type operators, called multilinear Hardy--Littlewood--Sobolev theorem on weighted variable Lebesgue spaces. Meanwhile, the weighted variable boundedness for the commutators of multilinear fractional type operators are also obtained. This generalizes some known work, such as Moen\cite{Moen} (2009), Bernardis--Dalmasso--Pradolini \cite{Ber2014} (2014), and Cruz-Uribe--Guzmán \cite{Cruz2020} (2020).
Another class of weights ${{\mathbb{A}}_{p( \cdot ),q(\cdot)}}$ are variable matrix weights that also characterized by certain fractional type operators. This generalize some previous results on matrix weights ${{\mathbb{A}}_{p( \cdot )}}$.
\end{abstract}

\maketitle
\tableofcontents
\section{\bf Introduction}\label{sec1}

The fundamental theory of harmonic analysis focuses on studying integral operators in various function spaces. Many classical linear and sublinear operators in harmonic analysis have multilinear counterparts that extend beyond the linear cases. The pioneering work on multilinear operators, initiated by Coifman and Meyer \cite{Coifman1975On,Coifman1978commutators}, has been studied by numerous mathematicians, including notable contributions from Christ and Journé \cite{Christ1987Polynomial}, Kenig and Stein \cite{Kenig1999Multilinear}, and Demeter, Tao, and Thiele \cite{Demeter2008Maximal}.

Let $b$ be a locally integrable function on $\mathbb{R}^{n}$ and $T$ be an linear operator. Then the commutator operator $T_b$ is defined as
$$T_{b}(f)=bT(f)-T(bf),$$
and $b$ is called the symbol function of $T_{b}$.  The pioneering work on $T_{b}$, where $T$ is a nonconvolution operator and $b\in {\rm BMO}(\mathbb{R}^{n})$, was initiated by Coifman, Rochberg, and Weiss \cite{Coifman1976Factorization}. They introduced a new characterization of the ${\rm BMO}(\mathbb{R}^{n})$ space through the boundedness of $T_{b}$. For any ball $B\subseteq \mathbb{R}^{n}$, we say $b\in{\rm BMO}(\mathbb{R}^{n})$ if the norm
$$\|b\|_{{\rm BMO}(\mathbb{R}^{n})}:=\sup_{B}\frac{1}{|B|}\int_{B}|b_{B}-b(x)|dx<\infty,$$
where $ b_{B}=\frac{1}{|B|}\int_{B}b(x)dx.$

The theory of commutators is crucial in studying partial differential equations, particularly in finding solutions to many elliptic PDEs. This paper focuses on the boundedness of corresponding commutators for integral operators, which is essential for understanding these solutions. 

Let $T$ be an $m$-sublinear operator and $\vec{b}=(b_1,\ldots,b_m)$ be locally integrable functions. The $m$-sublinear commutator generated by $T$ and $\vec{b}$ is defined as
\begin{equation*}
	{{T}_{\Sigma \vec b}}({f_1}, \ldots ,{f_m})(x) =\sum\limits_{j = 1}^m {{T}_{\vec b}^j} (\vec f)(x): = \sum\limits_{j = 1}^m {{T}({f_1}, \ldots ,({b_j}(x) - {b_j}){f_j}, \ldots {f_m})(x)}
\end{equation*}
and the iterated commutator, generated by  $T$ and $\vec{b}$, is defined by
\begin{equation*}
	{{T}_{\Pi \vec b}}(\vec f)(x) = {T}(({b_1}(x) - {b_1}){f_1}, \ldots ,({b_m}(x) - {b_m}){f_m})(x).
\end{equation*}

In 1999, Kenig and Stein \cite{Kenig1999Multilinear} introduced the following multilinear fractional operators
\begin{align}
	\mathscr{I}_{\alpha}(\vec{f})(x):=\int_{(\mathbb{R}^{n})^{m}}\frac{f_{1}(y_{1})\cdots f_{m}(y_{m})}{(|x-y_{1}|+\cdots+|x-y_{m}|)^{mn-\alpha}}d{\vec y},
\end{align}
where $0<\alpha<mn$. 
Clearly, ${{\I}_\alpha }(\vec f)(x)$ is absolutely convergent for any $\vec f \in {\left( {L_c^\infty } \right)^m}$. Additionally, they proved the boundedness of product of Lebesgue spaces. Moen \cite{Moen} investigated weighted inequalities for multilinear fractional integral operators, which greatly influenced this work.
In his study, the multilinear fractional maximal operator is defined as
\begin{align}\label{MFMO}
	{{\mathscr M}_\alpha }(\vec f)(x): = \mathop {\sup }\limits_{r > 0} \frac{1}{{|B(x,r){|^{m - \frac{\alpha }{n}}}}}\prod\limits_{i = 1}^m {\int_{B(x,r)} | {f_i}({y_i})|d{y_i}}.
\end{align}
Several researchers have recently investigated the boundedness of multilinear operators on certain function spaces. Chen and Xue \cite{Chen-Xue} focused on weighted estimates for a specific class of multilinear fractional type operators. Meanwhile, Chen and Wu \cite{chen-wu} examined multiple weighted estimates for commutators of multilinear fractional integral operators. Additionally, Xue \cite{xue7} explored weighted estimates for the iterated commutators of multilinear maximal and fractional type operators.

Recalling some definitions for the variable Lebesgue spaces, for a measurable subset $E \subseteq {\rn}$, we denote ${p_ - }(E)= \mathop {ess\inf }\limits_{x \in E} \{ p(x)\}, {p_ + }(E) =\mathop {ess\sup }\limits_{x \in E} \{ p(x)\}$. Especially, denote ${p_ - }={p_- }(\rn)$ and ${p_+ }={p_ + }(\rn)$. We give several sets of exponents as follows.
\begin{align*}
	&{\P}\left( E \right) = \{ p( \cdot ) :{\rm{E}} \to \left[ {1,\infty } \right) \text{ is measurable: } 1 < {p_ - }(E) \le {p_{\rm{ + }}}(E) < \infty \};\\
	&{{\P}_1}\left( {E} \right) = \{ p( \cdot ) :{\rm{E}} \to \left[ {1,\infty }\right) \text{ is measurable: } 1 \le {p_ - }(E) \le {p_{\rm{ + }}}(E) < \infty \};\\
	&{{\P}_0}\left( {E} \right) = \{ p( \cdot ) :{\rm{E}} \to \left[ {0,\infty }\right) \text{ is measurable: } 0 < {p_ - }(E) \le {p_{\rm{ + }}}(E) < \infty \}.
\end{align*}
Obviously, ${\P}\left( E \right) \subseteq {{\P}_1}\left( {E} \right) \subseteq {{\P}_0}\left( {E} \right)$. When $E=\mathbb{R}^n$, we take a shorthand. For example, ${\P}\left( \rn \right)$ will be denoted by ${\P}$.

Let $E \subseteq \rn$ and $p( \cdot ) \in \P_0(E)$. We define the variable exponent Lebesgue spaces with Luxemburg norm by
\begin{equation*}
	L_{}^{p( \cdot )}(E) = \{ f:{\left\| f \right\|_{L_{}^{p( \cdot )}(E)}} := \inf \{ \lambda  > 0:{\int_E^{} {\left( {\frac{{\left| {f(x)} \right|}}{\lambda }} \right)} ^{p(x)}}dx \le 1\}  < \infty \}.
\end{equation*}
For an open set $\Omega \subseteq \rn$, we define the locally variable exponent Lebesgue spaces by
\begin{equation*}
	L_{loc}^{p( \cdot )}(\Omega) = \{ f: \text{ For all compact subsets } E \subseteq \Omega, f \in L_{}^{p( \cdot )}(E) \}.
\end{equation*}
Let $\omega$ be a weight function on $E$. The variable exponent weighted Lebesgue spaces are defined by
\begin{equation*}
L_{}^{p( \cdot )}(E,\omega ) = \{ f:{\left\| f \right\|_{L_{}^{p( \cdot )}(E,\omega )}} = {\left\| {\omega f{\chi _E}} \right\|_{L_{}^{p( \cdot )}}} < \infty \}.
\end{equation*}
We say $p( \cdot ) \in LH(\rn)$(globally log-H\"{o}lder continuous functions), if $p( \cdot )$ satisfies
$$\left| {p(x) - p(y)} \right| \le \frac{C}{{ - \log \left( {\left| {x - y} \right|} \right)}}, \quad \text{when} \quad \left| {x - y} \right|< \frac{1}{2}, $$
there exist some ${p_\infty }>0$, such that for any $x \in \mathbb{R}^n$, then
$$\left| {p(x) - p_\infty} \right| \le \frac{C}{{\log \left( {e + \left| x \right|} \right)}}.
$$

In \cite{Cruz2017}, Cruz and Wang showed the following relationship:
\begin{align}
p(\cdot) \in LH \cap \P \text{    if and only if    } p'(\cdot) \in LH \cap \P.
\end{align}
The classical \( A_p \) and \( A_{p,q} \) weights theory, introduced by Muckenhoupt, studies the weighted \( L^p \) boundedness of Hardy-Littlewood maximal functions. For a detailed exploration of this theory, see \cite[Chapter 7]{Gra1}.
\begin{defn}[\cite{Gra1}]
Let \( 1 < p < \infty \). We say that a weight \( \omega \) belongs to \( A_p \) if it satisfies
	\begin{align*}
	\left[ \omega  \right]_{{A_p}}: = \mathop {\sup }\limits_{B \subseteq \rn} {\frac{1}{|B|}}{\left\| {{\omega ^{ - 1}}} \right\|_{{L^{p'}}(B)}}{\left\| \omega  \right\|_{L^p(B)}} < \infty.
	\end{align*}
	Denote $${A_\infty } = \bigcup\limits_{1 \le p < \infty } {{A_p}}.$$
\end{defn}
\begin{defn}[\cite{Gra1}]\label{RH1}
	A weight function \( \omega \) is said to belong to the reverse Hölder class \( RH_r \) if there exist two constants \( r > 1 \) and \( C > 0 \) such that the weight satisfies the following reverse Hölder inequality:
	$$\left(\frac{1}{|B|}\int_B \omega(x)^r\,dx\right)^{1/r}\le C\left(\frac{1}{|B|}\int_B \omega(x)\,dx\right)\quad\mbox{for every ball}\; B\subseteq \mathbb R^n.$$
	It is well established that if \( \omega \) belongs to \( A_p \) for some \( 1 < p < \infty \), then \( \omega \) also belongs to \( A_r \) for all \( r > p \), and \( \omega \) belongs to \( A_q \) for some \( 1 < q < p \). Additionally, if \( \omega \) is in \( A_{\infty} \), then there exists an \( r > 1 \) such that \( \omega \) belongs to the reverse Hölder class \( RH_r \).
\end{defn}

\begin{lem} [\cite{Cruz2020}]\label{lemma:Ainfty-prop}
Let \(\omega \in A_{\infty}\). Then, for each \(0 < \alpha < 1\), there exists $\beta \in (0,1)$ such that for any cube \(Q\) and any subset \(E \subseteq Q\) with \(\alpha|Q| \leq |E|\), it holds that \(\beta \omega(Q) \leq \omega(E)\). Similarly, for each \(0 < \gamma < 1\), there exists $\delta \in (0,1)$ such that if \(|E| \leq \gamma |Q|\), then \(\omega(E) \leq \delta \omega(Q)\).
\end{lem}

To address the boundedness of multilinear fractional-type operators in variable function spaces, we introduce a definition of variable multiple weights. 

\begin{defn}\label{vweight3}
Let  $p_i(\cdot) \in \P$, $i=1,\cdots,m$, with $\frac{1}{{p( \cdot )}} = \sum\limits_{i = 1}^m {\frac{1}{{{p_i}( \cdot )}}}$, $\frac{{{\alpha }}}{n}{\rm{ = }}\frac{1}{{{p}( \cdot )}}{\rm{ - }}\frac{1}{{{q}( \cdot )}} \in [0,m)$, and $\vec{\omega}=(\omega_1,\ldots,\omega_m)$ is a multiple weight with $\omega:= \prod\limits_{i = 1}^m {{\omega _i}}$. We say $\vec{\omega} \in A_{\vec{p}(\cdot),q(\cdot)}$, if it satisfies
	\begin{align*}
	{\left[ {\vec \omega } \right]_{{A_{\vec p( \cdot ),q( \cdot )}}}}: = \mathop {\sup }\limits_{B \subseteq \rn} {\left| B \right|^{\frac{\alpha }{n} - m}}{\left\| {\omega {\chi _B}} \right\|_{{L^{q( \cdot )}}}}\prod\limits_{i = 1}^m {{{\left\| {\omega _i^{ - 1}{\chi _B}} \right\|}_{{L^{{p_i}^\prime ( \cdot )}}}}}  < \infty.
	\end{align*}
\end{defn}

\begin{rem}
Actually, the above definition contain a broader range of weights, allowing us to derive established weights by imposing specific conditions.

	$(1)~$ If $\alpha=0$, then ${A_{\vec p( \cdot ),q( \cdot )}} = {A_{\vec p( \cdot )}}$ introduced in \cite{Cruz2020}.
	
	$(2)~$ If $\vec p( \cdot ) \equiv \vec p$, then ${A_{\vec p( \cdot ),q( \cdot )}} = {A_{\vec p,q}}$ introduced in \cite{Moen}.
	
	$(3)~$ If $\vec p( \cdot ) \equiv \vec p$ and $\alpha=0$, then ${A_{\vec p( \cdot ),q( \cdot )}} = {A_{\vec p}}$ introduced in \cite{Lerner}.
	
	$(4)~$ If $m=1$, then ${A_{\vec p( \cdot ),q( \cdot )}} = {A_{p( \cdot ), q(\cdot)}}$ introduced in \cite{Ber2014}.
	
	$(5)~$ If $m=1$ and $\alpha=0$, then ${A_{\vec p( \cdot ),q( \cdot )}} = {A_{p( \cdot )}}$ introduced in \cite{Cruz2011}. 
	
	$(6)~$ If $m=1$ and $\vec p( \cdot ) \equiv \vec p$, then ${A_{\vec p( \cdot ),q( \cdot )}} = {A_{ p, q}}$ introduced in \cite{Muckenhoupt1974}.
	
	$(7)~$ If $m=1$, $\vec p( \cdot ) \equiv \vec p$, and $\alpha=0$, then ${A_{\vec p( \cdot ),q( \cdot )}} = {A_p}$ introduced in \cite{Muckenhoupt1972}.
\end{rem}

\begin{figure}[!h]
	\begin{center}
		\begin{tikzpicture}
		\node (1) at(0,0) {$A_p$};
		\node (3) at(2.5,0) {$A_{p(\cdot)}$};
		\node (2) at(2.5,2.5) {$A_{p, q}$};
		\node (4) at(2.5,-2.5) {$A_{\vec{p}}$};
		\node (5) at(5,2.5) {$A_{p(\cdot), q(\cdot)}$};
		\node (6) at(5,0) {$A_{\vec{p}(\cdot)}$};
		\node (7) at(5,-2.5) {$A_{\vec{p}, q}$};
		\node (8) at(7.5,0) {$A_{\vec{p}(\cdot), q(\cdot)}$};
		\draw[->] (1)--(2);
		\draw[->] (1)--(3);
		\draw[->] (1)--(4);
		\draw[->] (2)--(5);
		\draw[->] (2)--(7);
		\draw[->] (3)--(5);
		\draw[->] (3)--(6);
		\draw[->] (4)--(6);
		\draw[->] (4)--(7);
		\draw[->] (5)--(8);
		\draw[->] (6)--(8);
		\draw[->] (7)--(8);
		\end{tikzpicture}
	\end{center}
	\caption{The relationships between weights}\label{figure1}
\end{figure}
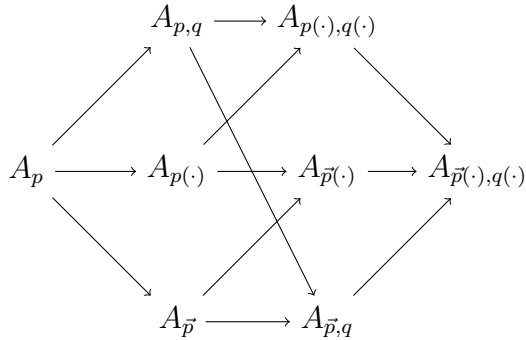

\begin{rem}
Let $\alpha  = \sum\limits_{i = 1}^m {{\alpha _i}}$, $\frac{{{\alpha _i}}}{n}{\rm{ = }}\frac{1}{{{p_i}( \cdot )}}{\rm{ - }}\frac{1}{{{q_i}( \cdot )}} \in (0,1)$, and $\frac{1}{{q( \cdot )}} = \sum\limits_{i = 1}^m {\frac{1}{{{q_i}( \cdot )}}}$. By H\"older's inequality, we have
	\begin{align*}
	{\left[ {\vec \omega } \right]_{{A_{\vec p( \cdot ),q( \cdot )}}}} \lesssim \prod\limits_{i = 1}^m {{{\left[ {{\omega _i}} \right]}_{{A_{{p_i}( \cdot ),{q_i}( \cdot )}}}}}.
	\end{align*}
Then we obtain
	$$\mathop  \prod \limits_{i = 1}^m {A_{{p_i}( \cdot ),{q_i}( \cdot )}} \subsetneqq {A_{\vec p( \cdot ),q( \cdot )}}.$$
\end{rem}

Additionally, we will present a crucial lemma \ref{vweight4} that supports this framework.
\begin{lem}\label{vweight4}
	Let  $p_i(\cdot) \in \P$, $i=1,\cdots,m$, with $\frac{1}{{p( \cdot )}} = \sum\limits_{i = 1}^m {\frac{1}{{{p_i}( \cdot )}}}$, and $\frac{{{\alpha }}}{n}{\rm{ = }}\frac{1}{{{p}( \cdot )}}{\rm{ - }}\frac{1}{{{q}( \cdot )}} \in [0,m)$. Then $\vec{\omega} \in A_{\vec{p}(\cdot),q(\cdot)}$ implies that 
	\begin{equation}
	\left\{
	\begin{aligned}
	&\omega _i^{-\frac{1}{m}} \in {A_{m{p_i}^\prime ( \cdot )}}, i=1,\dots,m,\\
	&\omega^{\frac{1}{m}}\in {A_{mq( \cdot )}}.
	\end{aligned}\right.
	\end{equation}
\end{lem}
\begin{proof}[Proof:]
	Note that $$\frac{1}{{mp( \cdot )}} + \sum\limits_{i \ne j} {\frac{1}{{m{p_i}^\prime ( \cdot )}}}  = \frac{1}{{(m{p_j}^\prime )'( \cdot )}}.$$
	By H\"older's inequality, we obtain
	\begin{align*}
	{\left\| {\omega _j^{ - \frac{1}{m}}{\chi _B}} \right\|_{{L^{m{p_j}^\prime ( \cdot )}}}}{\left\| {\omega _j^{\frac{1}{m}}{\chi _B}} \right\|_{{L^{(m{p_j}^\prime )'( \cdot )}}}}
	=& {\left\| {\omega _j^{ - \frac{1}{m}}{\chi _B}} \right\|_{{L^{m{p_j}^\prime ( \cdot )}}}}\left\| {{\omega ^{\frac{1}{m}}}\prod\limits_{i \ne j} {\omega _i^{ - \frac{1}{m}}} {\chi _B}} \right\|_{{L^{(m{p_j}^\prime )'( \cdot )}}}\\
	\lesssim& {\left\| {{\omega ^{\frac{1}{m}}}{\chi _B}} \right\|_{{L^{mp( \cdot )}}}}\prod\limits_{i = 1}^m {{{\left\| {\omega _i^{ - \frac{1}{m}}{\chi _B}} \right\|}_{{L^{m{p_i}^\prime ( \cdot )}}}}}\\
	\lesssim& \left| B \right|{\left( {{{\left| B \right|}^{\frac{\alpha }{n} - m}}\left\| {\omega {\chi _B}} \right\|_{{L^{q( \cdot )}}}^{}\prod\limits_{i = 1}^m {\left\| {\omega _i^{ - 1}{\chi _B}} \right\|_{{L^{{p_i}^\prime ( \cdot )}}}^{}} } \right)^{\frac{1}{m}}}.
	\end{align*}
	Thus, we have
	$$\left[ {\omega _j^{ - \frac{1}{m}}} \right]_{{A_{m{p_j}^\prime ( \cdot )}}}^m \lesssim {\left[ {\vec \omega } \right]_{{A_{\vec p( \cdot ),q( \cdot )}}}}.$$
	
	We also have the following estimate:
	\begin{align*}
	{\frac{1}{|B|}}{\left\| {{\omega ^{\frac{1}{m}}}{\chi _B}} \right\|_{{L^{mq( \cdot )}}}}{\left\| {{\omega ^{ - \frac{1}{m}}}{\chi _B}} \right\|_{{L^{(mq)'( \cdot )}}}} \lesssim& {\frac{1}{|B|}}{\left\| {{\omega ^{\frac{1}{m}}}{\chi _B}} \right\|_{{L^{mq( \cdot )}}}}\prod\limits_{i = 1}^m {{{\left\| {{\omega_i ^{ - \frac{1}{m}}}{\chi _B}} \right\|}_{{L^{m{q_i}^\prime ( \cdot )}}}}}\\
	\le& {\frac{1}{|B|}}{\left( {{{\left\| {\omega {\chi _B}} \right\|}_{{L^{q( \cdot )}}}}\prod\limits_{i = 1}^m {{{\left\| {{\omega_i ^{ - 1}}{\chi _B}} \right\|}_{{L^{{q_i}^\prime ( \cdot )}}}}} } \right)^{\frac{1}{m}}}\\
	\lesssim& {\left( {{{\left| B \right|}^{\frac{\alpha }{n} - m}}{{\left\| {\omega {\chi _B}} \right\|}_{{L^{q( \cdot )}}}}\prod\limits_{i = 1}^m {{{\left\| {{\omega_i^{ - 1}}{\chi _B}} \right\|}_{{L^{{p_i}^\prime ( \cdot )}}}}} } \right)^{\frac{1}{m}}},\\
	\end{align*}
	where $\frac{1}{{{q_i}( \cdot )}}: = \frac{1}{{{p_i}( \cdot )}} - \frac{\alpha }{{mn}}$, $i=1,\cdots,m$.
	
	Therefore, we also get
	$$\left[ {{\omega ^{\frac{1}{m}}}} \right]_{{A_{mq( \cdot )}}}^m \lesssim {\left[ {\vec \omega } \right]_{{A_{\vec p( \cdot ),q( \cdot )}}}}.$$
\end{proof}

The following lemma follows immediately from Lemma \ref{vweight4} and Lemma \ref{lemma:Ap-Ainf}.
\begin{lem}\label{cor:Ainfty}
    Let \( p_i(\cdot) \in \P \) for \( i = 1, \ldots, m \). Assume that $\frac{1}{{p( \cdot )}} = \sum\limits_{i = 1}^m {\frac{1}{{{p_i}( \cdot )}}}$ and $\frac{\alpha}{n} = \frac{1}{p(\cdot)} - \frac{1}{q(\cdot)} \in [0,m)$. If $\vec{\omega} \in A_{\vec{p}(\cdot), q(\cdot)}$, then both $u(\cdot):= \omega(\cdot)^{q(\cdot)}$ and $\sigma_j(\cdot) :=\omega_j(\cdot)^{-p_j'(\cdot)}$, for any $j=1,\ldots,m$, belong to $A_{\infty}$.
\end{lem}
\begin{rem}
	Note that the above definitions and results could have been given with the set of all balls $B \subseteq \rn$ replaced by the set of all cubes $Q \subseteq \rn$.
\end{rem}
The multilinear fractional-type operator ${{{\mathscr T}_{\alpha }}}$ can be either the multilinear fractional maximal operator ${{{\mathscr M}_{\alpha }}}$ or the multilinear fractional integral operator ${{{\mathscr I}_{\alpha }}}$. Weighted estimates of this operator have been established by Moen \cite{Moen} (2009), Chen and Xue \cite{Chen-Xue} (2010), and Chen and Wu \cite{chen-wu} (2013).

\hspace{-20pt}{\bf Theorem A} (\cite{Moen,Chen-Xue}).\label{FRLB1} {\it\
Let $1 < p_1, \ldots, p_m < \infty$,
	$\frac{1}{p} = \frac{1}{p_1} + \cdots + \frac{1}{p_m}$, $ \frac{1}{p} -\frac{1}{q} = \frac{\alpha}{n} \in (0,m)$, and $\vec \omega$ is a multiple weight with $\omega :=\prod\limits_{i = 1}^m {{\omega _i}}$. Then ${{{\mathscr T}_{\alpha }}}$ is bounded from ${L^{{p_1}}}({\omega _1}^{p_1}) \times  \cdots  \times {L^{{p_m}}}({\omega _m}^{p_m})$ to ${L^{q}}({\omega}^{q})$ if and only if $\vec \omega  \in {A_{\vec p,q}}$.
}

\hspace{-20pt}{\bf Theorem B} (\cite{Moen,Chen-Xue,chen-wu})\label{FRLB2}.{\it\
Let $1 < p_1, \ldots, p_m < \infty$,
	$\frac{1}{p} = \frac{1}{p_1} + \cdots + \frac{1}{p_m}$,
$ \frac{1}{p} -\frac{1}{q} = \frac{\alpha}{n} \in (0,m)$, and $\vec b \in {(\rm BMO)^m}$. Then for $\vec{\omega}\in A_{\vec{p}, q}$, 
	\begin{align}
		{\left\|\T_{\alpha}(\vec{f})\right\|}_{L^{q}(\omega^{q})} &\lesssim \prod_{i =1}^m{\big\|f_i\big\|}_{L^{p_i}(\omega_i^{p_i})};\notag\\
		{\left\| {{\T_{\alpha,\Pi \vec b}}(\vec f)} \right\|_{{L^q}(\omega^q)}} &\lesssim\prod\limits_{i = 1}^m {{{\left\| {{b_i}} \right\|}_{\rm BMO}}{{\left\| {{f_i}} \right\|}_{{L^{{p_i}}}(\omega _i^{{p_i}})}}};\notag\\
		{\left\| {\T_{\alpha,\sum {\vec b} }^i(\vec f)} \right\|_{{L^q}(\omega^q)}}&\lesssim{\left\| {{b_i}} \right\|_{\rm BMO}}\prod\limits_{i = 1}^m {{{\left\| {{f_i}} \right\|}_{{L^{{p_i}}}(\omega _i^{{p_i}})}}}.
	\end{align}
}
The following characterizations for variable weights $A_{p(\cdot)}$, $A_{p(\cdot),q(\cdot)}$, and $A_{\vec p(\cdot)}$ by maximal operators inspire us to study our main results.
 
In 2012, Cruz-Uribe, Fiorenza, and Neugebauer \cite{Cruz2012} firstly studied the characterization of $A_{p(\cdot)}$ by maximal operators $M$.

\hspace{-20pt}{\bf Theorem C} (\cite{Cruz2012}).\label{MAP} 
{\it\
Let $p( \cdot ) \in LH \cap \P$ and $w$ is a weight. Then $M$ is bounded on ${L^{p( \cdot )}}(\omega )$ if and only if $\omega \in A_{p( \cdot )}$.
}

In 2014, Bernardis, Dalmasso, and Pradolini \cite{Ber2014} characterized $A_{p(\cdot),q(\cdot)}$ by fractional maximal operators $M_\alpha$. 
 
\hspace{-20pt}{\bf Theorem D} (\cite{Ber2014}).\label{Apq} 
{\it\
Let $p( \cdot ) \in LH\cap \P$, $\frac{{{\alpha}}}{n}=\frac{1}{{{p}( \cdot )}}-\frac{1}{{{q}( \cdot )}} \in (0,1)$, and $\omega$ is a weight. Then $M_\alpha$ is bounded from ${{L^{p( \cdot )}}(\omega )}$ to ${{L^{q( \cdot )}}(\omega )}$ if and only if $\omega \in {A_{p( \cdot ),q( \cdot )}}$.
}

In 2020, Cruz-Uribe and Guzmán \cite{Cruz2020} demonstrated the following characterizations of $A_{\vec p(\cdot)}$ by multilinear maximal operators $\M$. 

\hspace{-20pt}{\bf Theorem E} (\cite{Cruz2020}).\label{m-M} 
{\it\
Let $p_i( \cdot ) \in LH\cap \P$, \( i = 1, \ldots, m \), with \( \frac{1}{p(\cdot)} = \sum_{i=1}^m \frac{1}{p_i(\cdot)} \), and \(\vec{\omega} = (\omega_1, \ldots, \omega_m)\) is a multiple weight with \(\omega = \prod_{i=1}^m \omega_i\). Then \(\mathscr{M}\) is  bounded from \( L^{p_1(\cdot)}(\omega_1) \times \cdots \times L^{p_m(\cdot)}(\omega_m) \) to \( L^{p(\cdot)}(\omega) \) if and only if \(\vec{\omega} \in A_{\vec{p}(\cdot)}\).}

It is natural to consider whether $A_{\vec p(\cdot),q(\cdot)}$ can be characterized by multilinear fractional-type operators $\T_\alpha$? The answer to this question is yes. To be more precise, one can see Theorems \ref{FMVLP} and \ref{FIVLP} below.

Throughout this paper, we use the following notations:

For a positive constant \( C \), independent of the relevant parameters, we write \( A \lesssim B \) to mean \( A \leq C B \), and \( A \approx B \) to indicate both \( A \lesssim B \) and \( B \lesssim A \). When \( C \) depends on \( \alpha \) and \( \beta \), we denote this by \( A \lesssim_{\alpha, \beta} B \).

Let \( E \subseteq X \) be an open set, and let \( p(\cdot): E \to [1, \infty) \) be a measurable function. The conjugate exponent \( p'(\cdot) \) is defined by
$
p'(\cdot) = \frac{p(\cdot)}{p(\cdot) - 1}.$

A weight \( \omega: X \to [0, \infty] \) is a locally integrable function such that \( 0 < \omega(x) < \infty \) for almost every \( x \in X \). The associated measure is \( d\omega(x) = \omega(x) d\mu(x) \), and we define the weighted measure of a set \( E \subseteq X \) as \( \omega(E) = \int_E \omega(x) \, d\mu(x) \).

For a set \( E \subseteq X \), a function \( f \), and a weight \( \sigma \), the weighted average is given by
\[
\avf_{\sigma, E} = \frac{1}{\sigma(E)} \int_E f \sigma \, d\mu.
\]
When \( \sigma = 1 \), this simplifies to \( \avf_E \).

The paper is organized as follows: In Section \ref{cha.mApq}, we present the first main result, which characterizes the new fractional-type variable multiple weights \( A_{\vec{p}(\cdot), q(\cdot)} \) and provides weighted estimates for multilinear fractional-type operators and their commutators. In Section \ref{cha.maApq}, we introduce and characterize the new fractional-type variable exponent matrix weights \( \mathbb{A}_{p(\cdot), q(\cdot)} \), which form the second main result. The complete proof of Theorem \ref{FMVLP} is given in Section \ref{longproof}.


\section{\bf New variable multiple weights $A_{\vec{p}(\cdot),q(\cdot)}$}\label{cha.mApq}
\subsection {Some classical results}
~

We first give some useful lemmas for variable exponent spaces.

\begin{lem}[\cite{Cruz2013}, Theorem 2.61]\label{lemma:fatou}
Let $p(\cdot) \in \P_0$. If $f \in L^{p(\cdot)}$ and ${f_k}$ converges to $f$ pointwise almost everywhere, then
$$
\|f\|_{L^{p(\cdot)}} \leq \liminf _{k \rightarrow \infty}\left\|f_k\right\|_{L^{p(\cdot)}}
$$
\end{lem}

\begin{lem}[\cite{Cruz2013}, Theorem 2.34]\label{lemma:dual}
Let $\pp \in \P$, then for every $f\in L^{p(\cdot)}$,
\[ \|f\|_{L^\pp} \approx \sup_{\|g\|_{L^\cpp} \leq 1}\int_\rn |fg|\,dx, \]
where the implicit constants depend only on $\pp$.
\end{lem}

\begin{lem}[\cite{Cruz2013}, Theorem 2.26]\label{Holder}
Let \(\pp \in \P(\Omega)\). If \(f\) belongs to \(L^{\pp}(\Omega)\) and \(g\) is in \(\Lcpp(\Omega)\), then
 $$
\int_{\Omega}|f(x) g(x)| d x \leq 4\|f\|_{L^{p(\cdot)}}\|g\|_{L^{p'(\cdot)}} .
$$
\end{lem}

\begin{lem}[\cite{Gu6}, Lemma 2.1]\label{Gu6}
	Let $p( \cdot ), p_i( \cdot ) $ belong to $ {\P}_0$, $ i = 1, \ldots ,m$, and $\frac{1}{{p(\cdot)}} = \sum\limits_{i = 1}^m {\frac{1}{{{p_i}(\cdot)}}}$. Then
	\begin{equation*}
		{\left\| {{f_1} \cdots {f_m}} \right\|_{{L^{p( \cdot )}}}} \lesssim \prod\limits_{i = 1}^m {{{\left\| {{f_i}} \right\|}_{{L^{{p_i}( \cdot )}}}}}.
	\end{equation*}
\end{lem}

\begin{lem} [\cite{Cruz2013}, Corollary 2.23] \label{lemma:mod-norm}
Let $\pp \in \P_1$.  If
	$\|f\|_\pp>1$, then
	\[ \rho_\pp(f)^{\frac{1}{p_+}}\leq \|f\|_\pp \leq 
	\rho_\pp(f)^{\frac{1}{p_-}}. \]
	If
	$\|f\|_\pp\leq 1$, then
	\[ \rho_\pp(f)^{\frac{1}{p_-}}\leq \|f\|_\pp \leq 
	\rho_\pp(f)^{\frac{1}{p_+}}. \]
 Hence, $\|f\|_\pp \lesssim 1$ if and only if $\rho_p(f) \lesssim 1$. 
\end{lem}

\begin{lem}[\cite{Cruz2013}, Lemma 3.4]\label{lemma:Ap-Ainf}
Let $\pp \in LH \cap \P_1$. If $\omega \in A_\pp$, then $u( \cdot ) = \omega {( \cdot )^{p( \cdot )}} \in A_\infty$.
\end{lem}

\begin{lem} [\cite{Cruz2013}, Lemma 3.5, 3.6]\label{lemma:p-infty-cond}
Let $\pp \in LH \cap \P_1$, $\omega \in A_\pp$, and
{$u(x)=\omega(x)^{p(x)}$}. For any cube $Q$ such that
	$\|\omega\chi_Q\|_\pp \geq 1$, then
	$\|\omega\chi_Q\|_\pp \approx u(Q)^{\frac{1}{p_\infty}}$.  Moreover,
	given any $E\subset Q$,
	$$
\frac{|E|}{|Q|} \lesssim\left(\frac{u(E)}{u(Q)}\right)^{\frac{1}{p_{\infty}}}
$$
	The implicit constants depend only on $\omega$ and $\pp$.
\end{lem}

\begin{rem}
	To apply Lemma~\ref{lemma:p-infty-cond}, note that by
	Lemma~\ref{lemma:mod-norm}, $\|\omega\chi_Q\|_\pp \geq 1$ if and only if
	$u(Q) \geq 1$. 
\end{rem}

\begin{lem}[\cite{Cruz2013}, Lemma 3.24]\label{lemma:diening}
	Let $\pp \in LH_0 \cap \P_0$.
	For any cube \( Q \),
$$
1 \lesssim |Q|^{p_{+}(Q)-p_{-}(Q)},
$$
where the implicit constant is determined solely by \(\pp\) and \(n\). Moreover, this inequality remains valid if \(p_+(Q)\) or \(p_-(Q)\) is replaced with \(p(x)\) for almost every \(x\) in \(Q\).
\end{lem}

The following lemma presents a weighted version of the Diening condition, as outlined in Lemma~\ref{lemma:diening}, which facilitates the application of the \( LH_0 \) condition.

\begin{lem}[\cite{Cruz2012}, Lemma 3.3] \label{lemma:lemma3.3}
	Let \(\pp \in LH \cap \P_1\). If \(\omega \in A_\pp\), then for all cubes \(Q\), we have
	$$
\left\|\omega \chi_Q\right\|_{p(\cdot)}^{p_{-}(Q)-p_{+}(Q)} \lesssim 1
$$
	the implicit constant in our calculations depends only on \(\pp\) and \(\omega\).
\end{lem}

\begin{lem}[\cite{Cruz2012}, (3.3)] \label{lemma:infty-bound}
Let \(\pp \in LH \cap \P_1\). If \(\omega \in A_\pp\), then there exists a constant \(t>1\), which depends only on \(\omega\), \(\pp\), and the dimension \(n\), such that
$$
\int_{\mathbb{R}^n} {\omega(x)^{p(x)}}{(e+|x|)^{-t n p_{-}}} d x \leq 1
$$
\end{lem}

\begin{lem}[\cite{Cruz2020}, Lemma 3.3] \label{lemma:p-infty-px}
	Let $\rr,\,\sst \in\P_0$, suppose 
	$$
|s(y)-r(y)| \leq \frac{C_0}{\log (e+|y|)}
$$
 Consider any set \(G\) and any non-negative measure \(\mu\). For each \(t \geq 1\), there exists a constant \(C = C(t, C_0)\) such that for all functions \(f\) satisfying \(|f(y)| \leq 1\)
	$$
\int_G|f(y)|^{s(y)} d \mu(y) \leq C \int_G|f(y)|^{r(y)} d \mu(y)+\int_G \frac{1}{(e+|y|)^{t n s_{-}(G)}} d \mu(y) .
$$
	Assuming instead that
$$
0 \leq r(y)-s(y) \leq \frac{C_0}{\log (e+|y|)},
$$
	then the same inequality holds for any function $f$.    
\end{lem}

Let \(q(\cdot)\) and \(p_i(\cdot)\) for \(i = 1, \ldots, m\) belong to \(\P_0\). Set $\frac{{{\alpha }}}{n}=\sum\limits_{i = 1}^m {\frac{1}{{{p_i}( \cdot )}}}-\frac{1}{{{q}( \cdot )}}>0$. For every cube \(Q\), we define \(\eta(Q)\) by
\begin{equation*}
\frac{1}{\eta(Q)}=\sum\limits_{i = 1}^m {\frac{1}{{{{({p_i})}_ - }(Q)}}},
\end{equation*}
and also define $\delta(Q)$ by
\begin{align}\label{delta}
 \frac{1}{{\delta (Q)}}: = \frac{1}{{\eta (Q)}} - \frac{\alpha }{n}.
\end{align}
It is easy to find that for a.e. $x \in Q$,
{\begin{align*}
\sum\limits_{i = 1}^m {\frac{1}{{{{({p_i})}_ - }}}}  - \frac{\alpha }{n} \ge \frac{1}{{\delta (Q)}}\ge \frac{1}{q_ -(Q)}\ge \frac{1}{q(x)}\ge \frac{1}{q_+},
\end{align*}
}
and this results that ${\delta}$ is a bounded function on all cubes.

The next lemma plays a crucial role in our proof.

\begin{lem}\label{q-relation}
Let $h(\cdot) \in \P$, $p_i(\cdot) \in LH \cap \P_1$, $i=1,\cdots,m$, with  $\frac{1}{{p( \cdot )}} = \sum\limits_{i = 1}^m {\frac{1}{{{p_i}( \cdot )}}}$, $\frac{{{\alpha }}}{n}=\frac{1}{{{p}( \cdot )}}-\frac{1}{{{q}( \cdot )}}>0$, and $v\in A_{h(\cdot)}$.
Suppose that $\delta(Q)$ is defined as \eqref{delta}, then
\begin{equation}\label{EQ-keyboundednorm}
\mathop {\sup }\limits_{Q\subseteq \rn}\sup\limits_{x\in Q}\|v^{-1}\chi_Q\|_{h'(\cdot)}^{\delta(Q)-q(x)}\lesssim_{{n,\alpha,\vec{p}(\cdot)}} 1.
\end{equation}               
\end{lem}

\begin{proof}
Without loss of generality, we assume that $\|v^{-1}\chi_Q\|_{h'(\cdot)}\leq 1$ and $m=2$. 

Let \(Q_0\) be the cube centered at the origin with volume \(|Q_0|=1\). For any cube \(Q\), we have either \(|Q| \leq |Q_0|\) or \(|Q| > |Q_0|\). {We prove inequality \eqref{EQ-keyboundednorm} for the first case; the second case follows by swapping \(Q\) and \(Q_0\). Let \(\ell(Q_0)\) be the side length of \(Q_0\). If \(\dist(Q, Q_0) \leq \ell(Q_0)\), then \(Q\) is contained in \(5 Q_0\).}
%
%
Then for any $x \in Q$,
\begin{equation}\label{EQ-logcontinuityof-q}
0\leq q(x)-\delta(Q)\lesssim_{n,\alpha ,{p_1}( \cdot ),{p_2}( \cdot )}(p_{1})_{+}(Q)-(p_{1})_{-}(Q)+(p_{2})_{+}(Q)-(p_{2})_{-}(Q).
\end{equation}
Therefore, by H\"older's inequality and $v\in A_{h'(\cdot)}$,
\begin{align} \label{eqn:small}
|Q|\lesssim5^{n}\| v^{-1}\chi_Q\|_{h'(\cdot)}|5Q_{0}|^{-1}\|v\chi_{5Q_{0}}\|_{h(\cdot)}\lesssim \|v^{-1}\chi_Q\|_{h'(\cdot)}\|v^{-1}\chi_{5Q_{0}}\|_{h'(\cdot)}^{-1}.
\end{align}
Hence, by \eqref{EQ-logcontinuityof-q}, \eqref{eqn:small}, and Lemma~\ref{lemma:diening},
 \begin{align*}
\|v^{-1}\chi_Q\|_{h'(\cdot)}^{\delta(Q)-q(x)}
 \leq \left(1+\|v^{-1}\chi_{5Q_{0}}\|_{h'(\cdot)}^{-1}\right)^{q_{+}- \mathop {\inf }\limits_{R \in {\mathbb{R}^n}} \delta (R)}|Q|^{\delta(Q)-q(x)}
\lesssim 1.
 \end{align*}

{Now assume \(\dist(Q, Q_0) \geqslant \ell(Q_0)\). Under this condition, there exists a cube \(\hat{Q}\) containing both \(Q\) and \(Q_0\). The side length \(\ell(\hat{Q})\) is comparable to \(\dist(Q, Q_0)\) and \(\dist(Q, 0)\), denoted as \(d_Q\).}
Therefore, by following the inequality~\(\eqref{eqn:small}\) and replacing \(5Q_0\) with \(\hat{Q}\), we obtain
   \begin{equation*}
|Q|\lesssim\vert\hat{Q}\vert\|v^{-1}\chi_Q\|_{h'(\cdot)}\|v^{-1}\chi_{\hat{Q}}\|_{h'(\cdot)}^{-1}.
   \end{equation*}
By the above, Lemma \ref{lemma:diening}, and the fact that
$\|v^{-1}\chi_{Q_0}\|_{h'(\cdot)} \leq \|v^{-1}\chi_{\hat{Q}}\|_{h'(\cdot)}$, we get
\begin{align}\label{es.v}
\|v^{-1}\chi_Q \|_{h'(\cdot)}^{\delta(Q)-q(x)}\lesssim
  \vert\hat{Q}\vert^{q(x)-\delta(Q)}.
\end{align}

{To estimate the final term effectively, observe that if \( p_j(\cdot) \) belongs to the \( LH \) class, there exist points \( x_1, x_2 \in \overline{Q} \) where the minimal values of \( p_1 \) and \( p_2 \) over \( Q \) are achieved. Specifically, let \( (p_1)_{-}(Q) = p_1(x_1) \) and \( (p_2)_{-}(Q) = p_2(x_2) \).}
Moreover, the distances \( |x_1| \) and \( |x_2| \) are approximately equal to the distance from cube Q to the origin, denote \( d_Q \). Hence, $\left|x_1\right|,\left|x_2\right| \approx d_Q$. Then by invoking the log-Hölder continuity of the functions \( p_1 \) and \( p_2 \), we can deduce that

\[ 
\left| {{\frac{1}{q_{\infty}}-\frac{1}{\delta(Q)}}} \right|
\leq
\left|\frac{1}{(p_1)_\infty}-\frac{1}{p_1(x_1)}\right|
+
\left|\frac{1}{(p_2)_\infty}-\frac{1}{p_2(x_2)}\right|
\lesssim 
\frac{1}{\log(e+d_{Q})}. \] 
Therefore, for $x\in Q$, since $|x|\approx d_{Q}$,
\[\left|{\frac{1}{q(x)}-\frac{1}{\delta(Q)}}\right|
\leq \left|{\frac{1}{q_{\infty}}-\frac{1}{\delta(Q)}}\right|+\left|{\frac{1}{p(x)}-\frac{1}{p_{\infty}}}\right|
\lesssim\frac{1}{\log(e+d_{Q})},
\]
Since  $\vert\hat{Q}\vert \lesssim (e+d_{Q})^{n}$, by \eqref{es.v}, we have
\begin{equation*}
\| v^{-1}\chi_Q\|_{h(\cdot)}^{\delta(Q)-q(x)}\lesssim \vert\hat{Q}\vert^{q(x)-\delta(Q)}\lesssim 1.
\end{equation*}
\end{proof}

\begin{lem}[\cite{Cruz2017}, Theorem 2.25]\label{prop:Ainfty-extrapol}
	Let $0<{p_0}<\infty$ and a weight $\omega_0\in A_\infty$, 
	\begin{align} \label{eqn:extrapol1}
		\|f\|_{L^{p_0}(\omega_0)} \lesssim \|g\|_{L^{p_0}(\omega_0)} 
	\end{align}
{Suppose for every pair \( (f, g) \in \mathcal{F} \), \( \|f\|_{L^{p_0}(\omega_0)} \) is finite, \( \pp \in \P_0 \), and there exists \( s \leq p_- \) such that \( \omega^s \in A_{\pp/s} \) and the maximal operator \( M \) is bounded on \( L^{(\pp/s)'}(\omega^{-s}) \). Under these conditions, for \( (f, g) \in \mathcal{F} \) with \( \|f\|_{L^\pp(\omega)} < \infty \), we have,}
 
	\[ \|f\|_{L^\pp(\omega)} \lesssim \|g\|_{L^\pp(\omega)}. \]
\end{lem}

Let $M^{\sharp}$ be  the standard sharp  maximal function of Fefferman and Stein, see \cite{Gra2}, and
\begin{align*}
	M^{\sharp}f (x)     &= \sup_{Q\ni x} \inf_{c \in \mathbb{R}} \dfrac{1}{|Q|} \int_{Q} |f(y)-c| \mathrm{d}y  \approx  \sup_{Q\ni x} \dfrac{1}{|Q|} \int_{Q} |f(y)-f_{Q}| \mathrm{d}y,
\end{align*}
where $f_Q = \langle f\rangle_Q$. The operators  $M_{\delta}$ and $M_{\delta}^{\sharp}$ are defined as  $M_{\delta}f (x)    = \big[ M(|f|^{\delta})(x) \big]^{1/\delta}$ and $M_{\delta}^{\sharp}f (x)    = \big[ M^{\sharp}(|f|^{\delta})(x) \big]^{1/\delta}$, respectively.

In order to manege our discussion, we define the weighted dyadic maximal operator $M_\sigma^{\D_0}$. For a given weight \(\sigma\), this operator is defined
\[ M_\sigma^{\D_0} f(x) = \sup_{Q\in \D_0} 
\frac{1}{\sigma(Q)}\int_Q |f|\sigma\,dy\cdot \chi_Q(x) 
= \sup_{Q\in \D_0} \langle |f|\rangle_{\sigma,Q} \chi_Q(x), \]
where the supremum is taken over all cubes in the collection $\D_0$ of
dyadic cubes:

\begin{lem}[\cite{Gra1}, (7.1.28)] \label{lemma:wtd-max-bound}
Let $1<p<\infty$, suppose $\sigma$ is a weight, then $M_\sigma^{\D_0}$ is bounded on ${L^p}(\sigma )$ and the operator norm is bounded by a constant depending only on $p$.
\end{lem}
\begin{lem}[\cite{Cruz2004}, Theorem 1.3]\label{FS1}
	Let $\omega \in A_\infty$ and $0<p,\gamma<\infty$, then
	\begin{align*}
		{\left\| f \right\|_{{L^p}(\omega )}} \lesssim {\left\| {M_\gamma ^{}f} \right\|_{{L^p}(\omega )}} \lesssim {\left\| {M_\gamma ^\# f} \right\|_{{L^p}(\omega )}}.
	\end{align*}
	
\end{lem}

\subsection{Main results}\label{frac}
~

The aim of this subsection is to build the characterizations for variable multiple weights $A_{\vec{p}(\cdot),q(\cdot)}$ by multilinear fractional-type operators $\T_{\alpha}$.

For $\alpha  \in \left[ {0,mn} \right)$ and any ball $B \subseteq \rn$, we define the multilinear fractional averaging operator by
$${\A_{\alpha,B}}(\vec f)(x): = {\left| B \right|^{\frac{\alpha }{n}}}\left( {\prod\limits_{i = 1}^m {{{\left\langle {{f_i}} \right\rangle }_B}} } \right){\chi _B}(x).$$ 
\begin{thm} \label{Cha.Apq}
	Let  $p_i(\cdot)\in \P, i=1,\ldots,m$ with $\frac{1}{{p( \cdot )}} = \sum\limits_{i = 1}^m {\frac{1}{{{p_i}( \cdot )}}}$, $\frac{{{\alpha }}}{n}=\frac{1}{{{p}( \cdot )}}-\frac{1}{{{q}( \cdot )}} \in [0,m)$, and $\vec{\omega}$ is a multiple weight, then  ${{\A_{\alpha,B}}}$ is bounded from ${L^{{p_1}( \cdot )}}({\omega _1}) \times  \cdots  \times {L^{{p_m}( \cdot )}}({\omega _m})$ to ${L^{q( \cdot )}}(\omega )$ uniformly for all balls $B$ if and only if $\vec\omega  \in {A_{\vec p( \cdot ),q( \cdot )}}$.
	Furthermore, we have 
	\begin{align*}
	{\left[ {\vec \omega } \right]_{{A_{\vec p( \cdot ),q( \cdot )}}}}{ \lesssim _{\vec p(\cdot)}} \mathop {\left\| {{\A_{\alpha,B}}} \right\|_{{L^{{p_1}( \cdot )}}({\omega _1}) \times  \cdots  \times {L^{{p_m}( \cdot )}}({\omega _m}) \to {L^{q( \cdot )}}(\omega )}} \le {\left[ {\vec \omega } \right]_{{A_{\vec p( \cdot ),q( \cdot )}}}},
	\end{align*}
where the implicit constant only depends on ${\vec p(\cdot)}$.
\end{thm}
\begin{proof}
	On the one hand, we can get
	\begin{align*}	    
	{\left\| {{\A_{\alpha ,B}}(\vec f)} \right\|_{{L^{q( \cdot )}}(\omega )}} &\le {\left| B \right|^{ - m + \frac{\alpha }{n}}}{\left\| {\omega {\chi _B}} \right\|_{{L^{q( \cdot )}}}}\prod\limits_{i = 1}^m {{{\left\| {\omega _i^{ - 1}{\chi _B}} \right\|}_{{L^{{p_i}^\prime ( \cdot )}}}}{{\left\| {{f_i}} \right\|}_{{L^{{p_i}( \cdot )}}({\omega _i})}}}\\
    &\le {\left[ {\vec \omega } \right]_{{A_{\vec p( \cdot ),q( \cdot )}}}}\prod\limits_{i = 1}^m {{{\left\| {{f_i}} \right\|}_{{L^{{p_i}( \cdot )}}({\omega _i})}}}.
	\end{align*}
	On the other hand, by Lemma \ref{lemma:dual}, there exists $g_j={{\omega_j h_j}}$, such that ${\left\| {g_j} \right\|_{{L^{{p_j}( \cdot )}}}} \le 1,j=1,\ldots,m$, we have 
	\begin{align*}
	&{\left| B \right|^{ - m + \frac{\alpha }{n}}}{\left\| {\omega {\chi _B}} \right\|_{{L^{q( \cdot )}}}}\prod\limits_{i = 1}^m {{{\left\| {\omega _i^{ - 1}{\chi _B}} \right\|}_{{L^{{p_i}^\prime ( \cdot )}}}}}\\
	\lesssim&_{\vec{p}(\cdot)} {\left\| {\omega {\chi _B}} \right\|_{{L^{q( \cdot )}}}}{\left| B \right|^{ - m + \frac{\alpha }{n}}}\prod\limits_{i = 1}^m {\int_B {{h_i}} }\\
	\le& {\left\| {{\A_{\alpha ,B}}} \right\|_{{L^{{p_1}( \cdot )}}({\omega _1}) \times  \cdots  \times {L^{{p_m}( \cdot )}}({\omega _m}) \to {L^{q( \cdot )}}(\omega )}}\prod\limits_{i = 1}^m {{{\left\| {{h_i}} \right\|}_{{L^{{p_i}( \cdot )}}({\omega _i})}}}.
	\end{align*}
	Thus,
	$${\left[ {\vec \omega } \right]_{{A_{\vec p( \cdot ),q( \cdot )}}}}{ \lesssim _{\vec p(\cdot)}}{\left\| {{\A_{\alpha ,B}}} \right\|_{{L^{{p_1}( \cdot )}}({\omega _1}) \times  \cdots  \times {L^{{p_m}( \cdot )}}({\omega _m}) \to {L^{q( \cdot )}}(\omega )}}.$$
	This finishes the proof.
\end{proof}


The question is whether the multilinear fractional maximal operators $\M_\alpha$ can characterize the weights $A_{\vec{p}(\cdot),q(\cdot)}$. Cruz-Uribe and Guzmán \cite[Theorem 2.4]{Cruz2020} have provided a positive answer when $\alpha=0$ (see {\bf Theorem E}). We now offer a positive answer for $\alpha>0$.

\begin{thm}\label{FMVLP}
Let $p_i( \cdot ) \in LH\cap \P$, $i = 1, \ldots ,m$, with $\frac{1}{{p( \cdot )}} = \sum\limits_{i = 1}^m {\frac{1}{{{p_i}( \cdot )}}}$, $\frac{{{\alpha}}}{n}=\frac{1}{{{p}( \cdot )}}-\frac{1}{{{q}( \cdot )}} \in [0,m)$, and $\vec \omega$ is a multiple weight with $\omega :=\prod\limits_{i = 1}^m {{\omega _i}}$. 
{Then ${{{\mathscr M}_{\alpha }}}$ is bounded from ${L^{{p_1}( \cdot )}}({\omega _1}) \times  \cdots  \times {L^{{p_m}( \cdot )}}({\omega _m})$ to ${L^{q( \cdot )}}(\omega)$ if and only if $\vec \omega  \in {A_{\vec p( \cdot ),q( \cdot )}}$.}
\end{thm}

The case $\alpha=0$ was covered by a recent result of [D. Cruz-Uribe and O. M. Guzmán;Publ. Mat. 2020]\cite{Cruz2020},and this theoerm represents the characterization for $\alpha > 0.$ 

{Combining Theorem 2.2 and $A_\infty$ extrapolation theorem of Rubio de Francia leads to the following conclusion.}
\begin{thm}[multilinear Hardy-Littlewood-Sobolev theorem on weighted variable exponents Lebesgue spaces]\label{FIVLP}
	Let  $p_i( \cdot ) \in LH\cap \P$, $i = 1, \ldots ,m$, with $\frac{1}{{p( \cdot )}} = \sum\limits_{i = 1}^m {\frac{1}{{{p_i}( \cdot )}}}$, $\frac{{{\alpha}}}{n}=\frac{1}{{{p}( \cdot )}}-\frac{1}{{{q}( \cdot )}} \in (0,m)$, and $\vec \omega$ is a multiple weight with $\omega :=\prod\limits_{i = 1}^m {{\omega _i}}$. Then ${{{\mathscr T}_{\alpha }}}$ is bounded from ${L^{{p_1}( \cdot )}}({\omega _1}) \times  \cdots  \times {L^{{p_m}( \cdot )}}({\omega _m})$ to ${L^{q( \cdot )}}(\omega)$ if and only if $\vec \omega  \in {A_{\vec p( \cdot ),q( \cdot )}}$.
\end{thm}	
\begin{proof}
	The necessity also follows immediately from Theorem \ref{Cha.Apq} or Theorem \ref{FMVLP} since 
	$${\A_{\alpha ,B}}(\vec f)(x) \lesssim {\M_\alpha }(\vec f)(x) \lesssim {\I_\alpha }(\left| {{f_1}} \right|, \cdots ,\left| {{f_m}} \right|)(x).$$
	When $\vec \omega  \in {A_{\vec p( \cdot ),q( \cdot )}}$, we prove the sufficiency as follows.
	
	In \cite[Proposition 3.2]{Chen-Xue}, Chen and Xue proved the following pointwise estimate, for any $\vec f \in {\left( {L_c^\infty } \right)^m}$,
	\begin{align*}
	{M}_{\frac{1}{m}}^\# ({\I_\alpha }(\vec f))(x) \lesssim {\M_\alpha }(\vec f)(x).
	\end{align*}
	By the above estimate and Theorem \ref{FMVLP}, we have 
	\begin{align*}
	{\left\| {{M}_{\frac{1}{m}}^\# ({\I_\alpha }(\vec f))} \right\|_{{L^{q( \cdot )}}(\omega )}}
	\lesssim {\left\| {{\M_\alpha }(\vec f)} \right\|_{{L^{q( \cdot )}}(\omega )}} 
	\lesssim \prod\limits_{i = 1}^m {{{\left\| {{f_i}} \right\|}_{{L^{{p_i}( \cdot )}}({\omega _i})}}}.
	\end{align*}
{where $M$ is the Hardy-Littlewood maximal operator.}
 
Since $\vec{\omega} \in A_{\vec{p}(\cdot), q(\cdot)}$, by Lemma \ref{vweight4}, we can choose $s = \frac{1}{m}$, ensuring ${\omega^s} \in A_{\frac{q(\cdot)}{s}}$ and ${\omega^{-s}} \in A_{\left(\frac{q(\cdot)}{s}\right)'}$. 
By the definition of $LH$, ${\left(\frac{q(\cdot)}{s}\right)}' \in LH$. Given $\frac{q(\cdot)}{s} \in \P$, it follows that ${\left(\frac{q(\cdot)}{s}\right)}'$ belongs to both $LH$ and $\P$. Consequently, $M$ is bounded on ${L^{{\left( \frac{q(\cdot)}{s} \right)}}({{\omega ^{-s}}})}$.
	
The extrapolation pairs are be given by
	\[ \F = \big\{ \big( \min(|{{{\I_{\alpha} (\vec f) }}}|,N)\chi_{B(0,N)},
	M_{\frac{1}{m}}^\#({{\I_{\alpha} (\vec f)}})\big) : \vec f \in {(L_c^\infty )^m}, N\in \n
	\big\}. \]
	Since $$\min(|{\I_{\alpha }(\vec f)}|,N)\chi_{B(0,N)} \in L^\infty_c \subseteq L^{q_0}(\omega_0)$$
for any \(q_0 > 0\) and a weight \(\omega_0\) in the \(A_\infty\) class, Lemma \ref{FS1} confirms that equation \(\eqref{eqn:extrapol1}\) holds for every pair in family \(\mathcal{F}\). This result is also supported by Lemma \ref{prop:Ainfty-extrapol}, which expands its applicability to \(A_\infty\) conditions.
	$$
	\min(|{\I_{\alpha }(\vec f)}|,N)\chi_{B(0,N)} \in L^\qq(\omega),
	$$
	and 
	$$
	\|\min(|{\I_{\alpha }(\vec f)}|,N)\chi_{B(0,N)} \|_{L^\qq(\omega)}
	\lesssim \| {M}_{\frac{1}{m}}^\#({\I_{\alpha }(\vec f)})\|_{L^\qq(\omega)}.
	$$
By Lemma $\ref{lemma:fatou}$ and the above estimates, we can deduce that
	\begin{align*}
	\|{{\I_{\alpha }(\vec f)}} \|_{L^\qq(\omega)}
	\le& \mathop {\lim \inf }\limits_{N \to \infty } \|\min(|{\I_{\alpha }(\vec f)}|,N)\chi_{B(0,N)} \|_{L^\qq(\omega)}\\
	\lesssim& \| {M}_{\frac{1}{m}}^\#({\I_{\alpha }(\vec f)})\|_{L^\qq(\omega)}\\
	\lesssim&\prod\limits_{i = 1}^m {{{\left\| {{f_i}} \right\|}_{{L^{{p_i}( \cdot )}}({\omega _i})}}}.
	\end{align*}
	
	Finally, the desired conclusion can follow from the limiting argument, since $L^\infty_c$ is dense in $L^{p_j(\cdot)}(\omega_j)$, $j=1,\dots,m$, see also \cite[Lemma 3.1]{Cruz2017}. 
\end{proof}

\begin{thm}\label{CFIVLP2}
	Let  $p_i( \cdot ) \in LH\cap \P$, $i = 1, \ldots ,m$, with $\frac{1}{{p( \cdot )}} = \sum\limits_{i = 1}^m {\frac{1}{{{p_i}( \cdot )}}}$, $\frac{{{\alpha}}}{n}=\frac{1}{{{p}( \cdot )}}-\frac{1}{{{q}( \cdot )}} \in (0,\frac{m}{t})$, and $\vec \omega  \in {A_{\vec p( \cdot ),q( \cdot )}}$ with ${{\vec \omega }^t} = (\omega _1^t, \cdots ,\omega _m^t) \in {A_{\frac{{\vec p( \cdot )}}{t},\frac{{q( \cdot )}}{t}}}$, for some $1 < t < \mathop {\min }\limits_{1 \le i \le m} \left\{ {p_i^ - } \right\}$. 
	If $\vec b \in {\left( {\BMO} \right)^m}$, then 
	\begin{align}
	{\left\| {{\T_{\alpha ,\sum {\vec b} }}} \right\|_{{L^{{p_1}( \cdot )}}({\omega _1}) \times  \cdots  \times {L^{{p_m}( \cdot )}}({\omega _m}) \to {L^{q( \cdot )}}(\omega )}} \lesssim& \sum\limits_{i = 1}^m {{{\left\| {{b_i}} \right\|}_{\BMO}}}.
	\end{align}
\end{thm}
\begin{proof}
	In \cite{Chen-Xue}, Chen and Xue proved the following pointwise estimate for ${\I_{\alpha ,\sum {\vec b} }}$. 
	Let $\delta  \in (0,\varepsilon ) \cap (0,\frac{1}{m}]$ such that for any $j=1,\cdots,m$,
	\begin{align*}
	M_\delta ^\# (\I_{\alpha ,\sum {\vec b} }(\vec f))(x){ \lesssim_\varepsilon }\left( {\sum\limits_{i = 1}^m {{{\left\| {{b_i}} \right\|}_{\BMO}}} } \right)\left( {\M_{L\log L,\alpha }^j(\vec f)(x) + {M_\varepsilon }({\I_\alpha }(\vec f))(x)} \right).
	\end{align*}
	By \cite[p. 1258]{Lerner} and \cite[(7.3)]{Chen-Xue},
	\begin{align}
	\M_{L\log L,\alpha}^j(\vec f)(x) \lesssim {\M_{L\log L,\alpha }}(\vec f)(x)
	\lesssim& \mathop {\sup }\limits_{x \in B} {\left( {{{\left| B \right|}^{\frac{{t\alpha }}{n}}}\prod\limits_{i = 1}^m {\frac{1}{{\left| B \right|}}\int_B {{{\left| {{f_i}} \right|}^{{t}}}} } } \right)^{\frac{1}{{{t}}}}} \notag\\
	=&{\left( {{\M_{t\alpha }}({{\left| {{f_1}} \right|}^t}, \cdots ,{{\left| {{f_m}} \right|}^t})(x)} \right)^{\frac{1}{t}}}.\label{LlogLa}
	\end{align}
	
Taking $\varepsilon  = \frac{1}{m}$ and combining Lemmas \ref{vweight4}, \ref{Gu6}, Theorem \ref{FIVLP}, and Theorem ${\bf C}$, we have
	\begin{align}
	&{\left\| {\M_{L\log L}^j(\vec f) + {{M}_{\frac{1}{m}} }({\I_\alpha }(\vec f))} \right\|_{{L^{q( \cdot )}}(\omega )}} \notag\\
	\lesssim&{\left\| {{{\left( {{\M_{t\alpha }}({{\left| {{f_1}} \right|}^t}, \cdots ,{{\left| {{f_m}} \right|}^t})(x)} \right)}^{\frac{1}{t}}}} \right\|_{{L^{q( \cdot )}}(\omega )}} + \left\| {M({{\left| {{\I_\alpha }(\vec f)} \right|}^{\frac{1}{m}}})} \right\|_{{L^{mq( \cdot )}}({\omega ^{\frac{1}{m}}})}^m \notag\\
	\lesssim&\left\| {{\M_{t\alpha }}({{\left| {{f_1}} \right|}^t}, \cdots ,{{\left| {{f_m}} \right|}^t})(x)} \right\|_{{L^{\frac{q( \cdot )}{t}}}(\omega^t )}^{\frac{1}{t}} + \left\| {{{\left| {{\I_\alpha }(\vec f)} \right|}^{\frac{1}{m}}}} \right\|_{{L^{mq( \cdot )}}({\omega ^{\frac{1}{m}}})}^m \notag\\
	\lesssim& \prod\limits_{i = 1}^m {\left\| {{{\left| {{f_i}} \right|}^t}} \right\|_{{L^{\frac{{{p_i}( \cdot )}}{t}}}({\omega _i}^t)}^{\frac{1}{t}}} +\prod\limits_{i = 1}^m {{{\left\| {{f_i}} \right\|}_{{L^{{p_i}( \cdot )}}({\omega _i})}}} \notag\\
	=&\prod\limits_{i = 1}^m {{{\left\| {{f_i}} \right\|}_{{L^{{p_i}( \cdot )}}({\omega _i})}}}. \label{similar}
	\end{align}
	By the extrapolation methods (see the proof of Theorem \ref{FIVLP}),
	\begin{align*}
	\|{\I_{\alpha,\sum {\vec b} }(\vec f)} \|_{L^\qq(\omega)}
	\lesssim \| {M}_\delta^\#({\I_{\alpha,\sum {\vec b}}}(\vec f))\|_{L^\qq(\omega)}
	\lesssim \left( {\sum\limits_{i = 1}^m {{{\left\| b \right\|}_{\BMO}}} } \right)\prod\limits_{i = 1}^m {{{\left\| {{f_i}} \right\|}_{{L^{{p_i}( \cdot )}}({\omega _i})}}}.
	\end{align*}

The proof can be done similarly using density.
\end{proof}

\begin{thm}\label{CFIVLP1}
	Let  $p_i( \cdot ) \in LH\cap \P$, $i = 1, \ldots ,m$, with $\frac{1}{{p( \cdot )}} = \sum\limits_{i = 1}^m {\frac{1}{{{p_i}( \cdot )}}}$, and $\frac{{{\alpha}}}{n}=\frac{1}{{{p}( \cdot )}}-\frac{1}{{{q}( \cdot )}} \in (0,m)$. 
	If $\vec b \in {\left( {\BMO} \right)^m}$, then
	\begin{align}
	{\left\| {{\T_{\alpha ,\prod {\vec b} }}} \right\|_{{L^{{p_1}( \cdot )}} \times  \cdots  \times {L^{{p_m}( \cdot )}} \to {L^{q( \cdot )}}}} \lesssim& \prod\limits_{i = 1}^m {{{\left\| {{b_i}} \right\|}_{\BMO}}}. 
	\end{align}
\end{thm}
\begin{proof}
	In \cite{xue7}, Xue also established the pointwise estimate for ${{\I_{\alpha ,\prod {\vec b} }}}$.
	Let $\delta  \in (0,\varepsilon ) \cap (0,\frac{1}{m}]$, such that
	\begin{align*}
	&\quad \quad{{M}_{\delta}^\# }({\I_{\alpha ,\prod {\vec b} }}(\vec f))(x)\\
	&\lesssim_{\varepsilon} \prod\limits_{i = 1}^m {{{\left\| {{b_i}} \right\|}_{\BMO}}} \left( {{\M_{L\log L,\alpha }}(\vec f)(x) + {{M}_\varepsilon }({\I_\alpha }(\vec f))(x)} \right)\\
    &+ \sum\limits_{j = 1}^{m - 1} {\sum\limits_{\sigma  \in C_j^m}^{} {\left( {\prod\limits_{i = 1}^j {{{\left\| {{b_{\sigma (i)}}} \right\|}_{\BMO}}} } \right){{M}_\varepsilon }({\I_{\prod {{{\vec b}_{\sigma '}}} }}(\vec f))(x)} } \\
	&:=J_1+J_2.
	\end{align*}
	By taking $r \in (1,\frac{{mn}}{\alpha })$, $\varepsilon \in (0,\frac{1}{m}]$, ${\alpha _i} = \frac{\alpha }{m}$, and $\frac{1}{{{q_i}( \cdot )}} = \frac{1}{{{p_i}( \cdot )}}{\rm{ - }}\frac{{{\alpha _i}}}{n}$, and using \eqref{LlogLa}, we obtain
	\begin{align}
	{\left\| {{J_1}} \right\|_{{L^{q( \cdot )}}(\omega )}} 
	\le& \left( {\prod\limits_{i = 1}^m {{{\left\| {{b_i}} \right\|}_{\BMO}}} } \right){\left\| {{{\left( {\M({{\left| {{f_1}} \right|}^r}, \cdots ,{{\left| {{f_m}} \right|}^r})} \right)}^{\frac{1}{r}}} + {M_\varepsilon }({\I_\alpha }(\vec f))} \right\|_{{L^{q( \cdot )}}}} \notag\\
	\lesssim& \left( {\prod\limits_{i = 1}^m {{{\left\| {{b_i}} \right\|}_{\BMO}}} } \right)\left( {\prod\limits_{i = 1}^m {\left\| {{{\left| {{f_i}} \right|}^r}} \right\|_{{L^{\frac{{{p_i}( \cdot )}}{r}}}}^{\frac{1}{r}} + {{\left\| {{\I_\alpha }(\vec f)} \right\|}_{{L^{q( \cdot )}}}}} } \right) \notag\\
	\lesssim& \prod\limits_{i = 1}^m {{{\left\| {{b_i}} \right\|}_{\BMO}}} {\left\| {{f_i}} \right\|_{{L^{{p_i}( \cdot )}}}}.  \notag
	\end{align}
	The term ${\| {{J_2}} \|_{{L^{q( \cdot )}}(\omega )}}$ can be simplified using Lemma \ref{FS1} and extrapolation methods, followed by repeated iteration.
	
	By the extrapolation methods (see the proof of Theorem \ref{FIVLP}),
	\begin{align*}
	{\left\| {{\I_{\alpha ,\prod {\vec b} }}(\vec f)} \right\|_{{L^{q( \cdot )}}}} \lesssim {\left\| {{M}_{\frac{1}{m}} ^\# ({\I_{\alpha ,\prod {\vec b} }}(\vec f))} \right\|_{{L^{q( \cdot )}}}} \lesssim \prod\limits_{i = 1}^m {{{\left\| {{b_i}} \right\|}_{\BMO}}} {\left\| {{f_i}} \right\|_{{L^{{p_i}( \cdot )}}}}.
	\end{align*}
The proof can be accomplished from the density.
\end{proof}

\begin{rem}
It is obvious to find that the results of Theorems \ref{FMVLP} and \ref{FIVLP} generalized Theorems {\bf A--E}.
\end{rem}

\section{\bf New variable matrix weights ${{\mathbb{A}}_{p( \cdot ),q( \cdot )}}$}\label{cha.maApq}
In this section, we will develop characterizations for the newly introduced variable matrix weights, denoted as ${{\mathbb{A}}_{p( \cdot ),q( \cdot )}}$, and introduce several notations and concepts related to the theory of matrix weights.

Define \(\Lpp(\Omega; \R^d)\) as the set of vector-valued, measurable functions \(\vf : \Omega \to \R^d\) such that
\[\| \vf\|_{\Lpp(\Omega;\R^d)} := \| |\vf|\|_{\Lpp(\Omega)} <\infty.\]

A matrix function maps a set \(\Omega \subseteq \mathbb{R}^n\) to the collection of \(d \times d\) matrices. This function is considered measurable if each of its components is a measurable function.

Let \(\calS_d\) denote the collection of \(d \times d\) self-adjoint and positive semi-definite matrices. A matrix weight is a measurable matrix function \(W : \Omega \to \calS_d\), where \(\left\| W \right\| \in L_{{\rm loc}}^1(\Omega)\), meaning the eigenvalues of \(W\) are locally integrable. A matrix weight is invertible if it is positive definite almost everywhere.

Let \( W: \Omega \to \mathcal{S}_d \) is a matrix weight, we define

\[{L^{p( \cdot )}}(W,\Omega )=\left\{ {{\bf{f}}:\Omega \to \mathbb{R}^d:{{\left\| {\bf{f}} \right\|}_{{L^{p( \cdot )}}(W,\Omega )}}: = {{\left\| {\left| {W{\bf{f}}} \right|} \right\|}_{{L^{p( \cdot )}}(\Omega )}} < \infty } \right\}.\]

In order to study the matrix weighted boundedness of fractional-type operators, we establish the following matrix weights.
\begin{defn}
Let $\pp, \qq \in \P(\Omega)$ with $\frac{1}{{p( \cdot )}} - \frac{1}{{q( \cdot )}} = \frac{\alpha }{n} \in \left( {0,1} \right)$, and $W:\Omega \to {\calS}_d$ be a matrix weight. We say $W \in {{\bbA}_{p( \cdot ),q( \cdot )}}$, if it satisfies
	\begin{align*}
{\left[ W \right]_{{{\mathbb{A}}_{p( \cdot ),q( \cdot )}}}} := \mathop {\sup }\limits_{Q \subseteq \Omega } {\left| Q \right|^{\frac{\alpha }{n}-1}}{\left\| {{{\left\| {\left\| {{W^{ - 1}}(y)W(x)} \right\|{\chi _Q}(y)} \right\|}_{{L_{y}^{p'( \cdot )}}(\Omega )}}{\chi _Q}(x)} \right\|_{L_x^{q( \cdot )}(\Omega )}}< \infty.
	\end{align*}
\end{defn}

When \(\alpha = 0\), we have \({{\bbA}_{p(\cdot), q(\cdot)}} = {{\bbA}_{p(\cdot)}}\), as established by Cruz-Uribe and Penrod in \cite{Cruz2023}. In this context, \(L_x^{p(\cdot)}(\Omega)\) indicates that the integration in the norm is performed with respect to the variable \(x\). Similar to scalar weights, when \(\Omega = \mathbb{R}^n\), an equivalent definition can be obtained by replacing cubes with balls, though for more complex domains, transitioning between the two may be necessary. Thus, for any matrix weight \(W\) defined on a domain \(\Omega\), it is implicitly assumed that \(W\) satisfies the \(\mathbb{A}_{p(\cdot), q(\cdot)}\) matrix condition in an extended domain \(\Omega'\).

For any finite set of vector-valued functions $\{\vf_i\}_{i=1}^m \subseteq \Lpp(\Omega;\R^d)$, we have
\begin{equation}\label{SumNorm:equiv}
\frac{1}{m} \sum_{i=1}^m\left\|\mathbf{f}_i\right\|_{L^{p(\cdot)}\left(\Omega, \mathbb{R}^d\right)} \leq\left\|\sum_{i=1}^m\left|\mathbf{f}_i\right|\right\|_{L^{p(\cdot)}(\Omega)} \leq \sum_{i=1}^m\left\|\mathbf{f}_i\right\|_{L^{p(\cdot)}\left(\Omega, \mathbb{R}^d\right)}
\end{equation}
The upper boundedness is immediate from the triangle inequality. The lower
boundedness holds since for each $i$, $|\vf_i(x)| \leq \sum_{j=1}^m
|\vf_j(x)|$.

Recall that the operator norm of a $d\times d$ matrix $A$ is defined by
$$\left\| A \right\| = \mathop {\sup }\limits_{\vv \in {\R^d},|\vv| = 1} \left| {A\vv} \right|.$$

If \(W \in \mathcal{S}_d\), it has \(d\) non-negative eigenvalues, denoted by \(\lambda_i\) for \(1 \leq i \leq d\), and there exists an orthogonal matrix \(U\) such that \(U^T W U\) is diagonal. This diagonal matrix can be represented as \(D(\lambda_1, \ldots, \lambda_d)\) or simply \(D(\lambda_i)\). Moreover, when \(W\) is a measurable matrix function with values in \(\mathcal{S}_d\), the corresponding matrix function \(U\) can also be chosen to be measurable.

\begin{lem}[\cite{Cruz2023}, Proposition 3.2]\label{VectorDuality}
	Let \(\pp \in \P(\Omega)\). Then for any \(\vf \in \Lpp(\Omega; \mathbb{R}^d)\), there exists \(\vg \in \Lcpp(\Omega; \mathbb{R}^d)\) with \(\|\vg\|_{\Lcpp(\Omega; \mathbb{R}^d)} \leq d\) such that
	\[ \|\vf\|_{\Lpp(\Omega;\R^d)} \leq 4 \int_{\Omega} \vg(x) \cdot \vf(x) \, dx.\]
\end{lem}

\begin{lem}[\cite{roudenko_matrix-weighted_2002}, Lemma 3.2]\label{opNorm:equiv}
	If $\{\ve_1, \ldots, \ve_t\}$ is any orthonormal basis in $\R^t$, then for any $d\times t$ matrix $V$, we have
	\begin{align*}
		\left\| V \right\| \approx \sum\limits_{i = 1}^t {\left| {V{\ve_i}} \right|},
	\end{align*}
	with implicit constants depending only on $t$.
\end{lem}

Second, recall that for any two self-adjoint $d\times d$ matrices $V$ and $W$, we have
\begin{align*}
	\left\| {VW} \right\| = \left\| {{{(VW)}^*}} \right\| = \left\| {{W^*}{V^*}} \right\| = \left\| {WV} \right\|.
\end{align*}
\begin{lem}[\cite{bownik_extrapolation_2022}, Theorem 4.11]\label{EllipsoidApprox}
Let \( r: \mathbb{R}^n \times \mathbb{R}^d \to [0, \infty) \) be any measurable norm function. Then, there exists a  measurable matrix function \( W: \Omega \to \mathcal{S}_d \) that ensures for any point \( x \in \mathbb{R}^n \) and vector \( \mathbf{v} \in \mathbb{R}^d \),
$$
r(x, \mathbf{v}) \leq|W(x) \mathbf{v}| \leq \sqrt{d} r(x, \mathbf{v})
$$
\end{lem}

Given a matrix weight \(W\) and a cube \(Q \subseteq \Omega\), we define the associated reducing operator, which can be viewed as an \(\Lpp(\Omega)\) average of \(W\) on \(Q\). For each \(x \in \Omega\), define the norm \(r(x, \cdot): \mathbb{R}^d \to [0, \infty)\) as \(r(x, \vv) = |W(x)\vv|\), and let \(r^*(x, \cdot)\) be its dual norm, defined by \(r^*(x, \vv) = |W^{-1}(x) \vv|\). 

For a specified cube \(Q\), we define \(\langle r \rangle_{\pp,Q}: \mathbb{R}^d \to [0, \infty)\) by
\begin{equation}\label{AvgNorm}
	\langle r\rangle_{\pp,Q} (\vv)
	:= 1/|Q|^{1/p_Q} \|\chi_Q(\cdot)r(\cdot,\vv)\|_{\Lpp(\Omega)}
	\approx_{d} 1/|Q|^{1/p_Q} \| \chi_Q(\cdot)\vv\|_{\Lpp(W,\Omega)},
\end{equation}
where $p_Q$ is the harmonic mean of $\pp$ on $Q$, defined by
$$\frac{1}{{{p_Q}}} = \frac{1}{{\left| Q \right|}}\int_Q {\frac{{dx}}{{p(x)}}}.$$

Since \(\|\cdot\|_{\Lpp(W, \Omega)}\) is a norm, it follows that \(\langle r \rangle_{\pp, Q}\) defines a norm on \(\mathbb{R}^d\). By Lemma \ref{EllipsoidApprox}, there exists a positive-definite, self-adjoint (constant) matrix \(W_Q^{q(\cdot)}\) such that \(\langle r \rangle_{{q(\cdot)}, Q}(\vv) \approx |W_Q^{q(\cdot)} \vv|\) for all \(\vv \in \mathbb{R}^d\). This matrix \(W_Q^{q(\cdot)}\) is called the reducing operator associated with \(r\) on the cube \(Q\). Similarly, the reducing operator associated with \(r^*\) on \(Q\), denoted \(\overline{W}_Q^\cpp\), satisfies \(\langle r^* \rangle_{\cpp, Q}(\vv) \approx |\overline{W}_Q^\cpp \vv|\). Cruz and Penrod \cite{Cruz2023} used reducing operators to provide an equivalent characterization of \(\bbA_{p(\cdot)}\). Here, we also employ reducing operators to characterize \({{\bbA}_{p(\cdot), q(\cdot)}}\).

\begin{pro}\label{ApqR}
	Let $\pp, \qq \in \P(\Omega)$ with $\frac{1}{{p( \cdot )}} - \frac{1}{{q( \cdot )}} = \frac{\alpha }{n} \in \left( {0,1} \right)$, and $W:\Omega \to {\calS}_d$ be a matrix weight. Then $\bbA_{p( \cdot ),q( \cdot )} = \bbA_{p( \cdot ),q( \cdot )}^R$. More precisely, ${\left[ W \right]_{\bbA_{p( \cdot ),q( \cdot )}^R}} \approx {\left[ W \right]_{\bbA_{p( \cdot ),q( \cdot )}^{}}}$, where the implicit constants are dependent only on $d$ and $${\left[ W \right]_{\bbA_{p( \cdot ),q( \cdot )}^R}}: = \mathop {\sup }\limits_{Q \subseteq \Omega } \left\| {W_Q^{q( \cdot )}\overline{W}_Q^{p'( \cdot )}} \right\|.$$
\end{pro}
\begin{proof}[Proof:]
	Let $\{\ve_i\}_{i=1}^d$ be the orthonormal basis in $\R^d$. By Lemma~\ref{opNorm:equiv}, inequality \eqref{SumNorm:equiv}, and the definition of the reducing operators $\overline{W}_Q^\cpp$ and $W_Q^{q(\cdot)}$, we have
	\begin{align*}
		&{\left| Q \right|^{\frac{\alpha }{n} - 1}}{\left\| {{{\left\| {\left\| {W(x){W^{ - 1}}( \cdot )} \right\|{\chi _Q}( \cdot )} \right\|}_{{L^{p'( \cdot )}}(\Omega )}}{\chi _Q}(x)} \right\|_{L_x^{q( \cdot )}(\Omega )}}\\
		\approx& {\left| Q \right|^{\frac{\alpha }{n} - 1}}{\left\| {{{\left\| {\sum\limits_{i = 1}^d {\left| {{W^{ - 1}}( \cdot )W(x){\ve_i}} \right|} {\chi _Q}( \cdot )} \right\|}_{{L^{p'( \cdot )}}(\Omega )}}{\chi _Q}(x)} \right\|_{L_x^{q( \cdot )}(\Omega )}}\\
		\approx&{{\left| Q \right|}^{\frac{\alpha }{n} - 1}}\sum\limits_{i = 1}^d {{{\left\| {{{\left\| {\left| {{W^{ - 1}}( \cdot )W(x){\ve_i}} \right|{\chi _Q}( \cdot )} \right\|}_{{L^{p'( \cdot )}}(\Omega )}}{\chi _Q}(x)} \right\|}_{L_x^{q( \cdot )}(\Omega )}}}\\
		\approx&{{\left| Q \right|}^{\frac{\alpha }{n} - \frac{1}{{{p_Q}}}}}\sum\limits_{i = 1}^d {{{\left\| {\left| {\overline{W}_Q^{p'( \cdot )}W(x){\ve_i}} \right|{\chi _Q}(x)} \right\|}_{L_x^{q( \cdot )}(\Omega )}}}\\
		\approx&{\left| Q \right|^{\frac{\alpha }{n} - \frac{1}{{{p_Q}}}}}{\left\| {\left\| {\overline{W}_Q^{p'( \cdot )}W(x)} \right\|{\chi _Q}(x)} \right\|_{L_x^{q( \cdot )}(\Omega )}}\\
		\approx&{\left| Q \right|^{\frac{\alpha }{n} - \frac{1}{{{p_Q}}}}}{\left\| {\sum\limits_{i = 1}^d {\left| {W(x)\overline{W}_Q^{p'( \cdot )}{\ve_i}} \right|} {\chi _Q}(x)} \right\|_{L_x^{q( \cdot )}(\Omega )}}\\
		\approx&\sum\limits_{i = 1}^d {{{\left| Q \right|}^{\frac{\alpha }{n} - \frac{1}{{{p_Q}}}}}{{\left\| {\left| {W(x)\overline{W}_Q^{p'( \cdot )}{\ve_i}} \right|{\chi _Q}(x)} \right\|}_{L_x^{q( \cdot )}(\Omega )}}}\\
		\approx&\sum\limits_{i = 1}^d {{{\left| Q \right|}^{\frac{\alpha }{n} - \frac{1}{{{p_Q}}} + \frac{1}{{{q_Q}}}}}\left| {\overline{W}_Q^{q( \cdot )}\overline{W}_Q^{p'( \cdot )}{\ve_i}} \right|}\\
		\approx& \left\| {\overline{W}_Q^{q( \cdot )}\overline{W}_Q^{p'( \cdot )}} \right\|,
	\end{align*}
	where the fifth step holds since $\overline{W}_Q^\cpp$ and $W$ are self-adjoint and we can commute them within the operator norm.
\end{proof}

The fractional averaging operator ${\A_{\alpha ,Q}}$ is given by
$${\A_{\alpha ,Q}}(\vf)(x) = \left( {{{\left| Q \right|}^{ \frac{\alpha}{n}- 1}}\int_Q {\vf(y)dy} } \right){\chi _Q}(x).$$

The fractional Christ-Goldberg maximal operator is defined by the following expression
$${{\calM}_{\alpha ,W}}(\vf)(x) := \mathop {\sup }\limits_{x \in Q} {\left| Q \right|^{\frac{\alpha }{n} - 1}}\int_Q {\left| {W(x){W^{ - 1}}(y)\vf(y)} \right|dy}.$$

The motivation for the definition of the fractional Christ-Goldberg maximal operator stems from an observation related to linear operators. Specifically, if \(T\) is a linear operator, then it satisfies the condition $${\left\| {T(\vf)} \right\|_{{L^{q(\cdot)}}(W,\Omega)}} \le C{\left\| \vf \right\|_{{L^{p(\cdot)}}(W,\Omega)}}$$ if and only if the operator $T_W$,
satisfies \[{\left\| {T_W(\vf)} \right\|_{{L^{q(\cdot)}}(\Omega)}} \le C{\left\| \vf \right\|_{{L^{p(\cdot)}}(\Omega)}},\]
where $ T_W\vf(x) = W(x)T(W^{-1}\vf)(x).$

Obviously, we have 
$$\mathop {\sup }\limits_{Q \subseteq \Omega} \left| {{\A_{\alpha ,Q,W}}(\vf)} \right| \le {\calM_{\alpha ,W}}(\vf)(x).$$

We present the main results of this section as follows which characterizes the matrix weights ${{\bbA}_{p( \cdot ),q( \cdot )}}$.

	\begin{thm}\label{cha.Apq}
Let \(\pp, \qq \in \P(\Omega)\) with the relationship \(\frac{1}{{p( \cdot )}} - \frac{1}{{q( \cdot )}} = \frac{\alpha }{n}\) in \((0,1)\), the operators \({\A_{\alpha ,Q}}\) are uniformly bounded from \({L^{p( \cdot )}}(W,\Omega )\) to \({L^{q( \cdot )}}(W,\Omega )\) uniformly for cubes in \(\Omega\) if and only if \(W\) belongs to \({\bbA_{p( \cdot ),q( \cdot )}}\).
		Moreover, 
		$${\left[ W \right]_{{\bbA}_{p( \cdot ),q( \cdot )}}} \lesssim_d \mathop {\sup }\limits_{Q \subseteq \Omega } {\left\| {{\A_{\alpha ,Q}}} \right\|_{{L^{p( \cdot )}}(W,\Omega ) \to {L^{q( \cdot )}}(W,\Omega )}} \le 4{\left[ W \right]_{{{\bbA}_{p( \cdot ),q( \cdot )}}}}.$$
	\end{thm}

\begin{proof}
	Sufficiency: by Lemma \ref{Holder}, 
	\begin{align*}
		&{\left\| \left| {{\A_{\alpha ,Q}}(\vf)} \right| \right\|_{{L^{q( \cdot )}}(W,\Omega )}}\\
		\le& 1/{\left| Q \right|}^{ 1-\frac{\alpha }{n}}{\left\| {\left( {\int_Q {\left| {\vf(y)W(y)W(x){W^{ - 1}}(y)} \right|dy} } \right){\chi _Q}(x)} \right\|_{L_x^{q( \cdot )}(\Omega )}}\\
		\le& 1/{\left| Q \right|}^{ 1-\frac{\alpha }{n}}{\left\| {\left( {\int_Q {\left\| {W(x){W^{ - 1}}(y)} \right\|\left| {W(y)\vf(y)} \right|dy} } \right){\chi _Q}(x)} \right\|_{L_x^{q( \cdot )}(\Omega )}}\\
		\le& 4/{\left| Q \right|}^{ 1-\frac{\alpha }{n}}{\left\| {\left( {{{\left\| {\left\| {W(x){W^{ - 1}}(y)} \right\|{\chi _Q}(y)} \right\|}_{L_y^{p'( \cdot )}(\Omega )}}{{\left\| {\vf(y)W(y)} \right\|}_{L_y^{p( \cdot )}(\Omega )}}} \right){\chi _Q}(x)} \right\|_{L_x^{q( \cdot )}(\Omega )}}\\
		\le& 4{\left[ W \right]_{{\bbA_{p( \cdot ),q( \cdot )}}}}{\left\| {\vf} \right\|_{L_{}^{p( \cdot )}(W,\Omega )}}.
	\end{align*}
	
	Necessity:
	By Proposition \ref{ApqR}, we shall show that ${\left[ W \right]_{{\bbA}_{p( \cdot ),q( \cdot )}^R}} < \infty$. Fix a cube  $Q\subseteq \Omega$ and a $\vv \in \R^d$
	with $|\vv|=1$. By Lemma \ref{VectorDuality}, the self-adjointness of $W^{-1}$ and $W_Q^\pp$, and the linearity of the dot product, we have that for some $\vg \in \Lpp(\Omega;\R^d)$ with $\|\vg\|_{\Lpp(\Omega;\R^d)}\leq d$,
	\begin{align*}
		|\overline{W}_Q^\cpp W_Q^{q(\cdot)} \vv| 
		&\approx_d
		|Q|^{-1/p'_Q} \||W^{-1}(\cdot)W_Q^{q(\cdot)} \vv|\chi_Q(\cdot)\|_{\Lcpp(\Omega)}\\
		& \le4 |Q|^{-1/p'_Q} \int_Q W^{-1}(y) W_Q^{q(\cdot)} \vv \cdot \vg(y) \,dy\\
		& =4 |Q|^{-1/p'_Q}  \int_Q \vv \cdot W_Q^{q(\cdot)} W^{-1}(y) \vg(y) \,dy\\
		& =4 |Q|^{-1/p'_Q} \vv \cdot W_Q^{q(\cdot)} \int_Q W^{-1}(y) \vg(y)\, dy\\
		& \le4 |Q|^{-1/p'_Q} \left| W_Q^{q(\cdot)} \int_Q  \vg(y)W^{-1}(y) \,dy\right|\\
		&\lesssim_d |Q|^{{\frac{\alpha }{n} - 1}}\left\| W(x) \int_Q \vg(y)W^{-1}(y) dy\;
		\chi_Q(x)\right\|_{{L_x^{q( \cdot )}(\Omega )}} \\
		&= \|W(x)\A_{\alpha,Q}(\vg W^{-1})(x)\|_{{L_x^{q( \cdot )}(\Omega )}}\\
		&\le {\left\| {{\A_{\alpha ,Q}}} \right\|_{L^{p( \cdot )}(W,\Omega ) \to {L^{q( \cdot )}}(W,\Omega )}} {\left\| {\vg {W^{ - 1}}} \right\|_{{L^{p( \cdot )}}(W,\Omega )}}\\
		&\lesssim_d {\left\| {{\A_{\alpha ,Q}}} \right\|_{L^{p( \cdot )}(W,\Omega ) \to {L^{q( \cdot )}}(W,\Omega )}},
	\end{align*}
	where the third step holds due to the self-adjointness.

	By Lemma \ref{opNorm:equiv},
	$${\left[ W \right]_{{\bbA}_{p( \cdot ),q( \cdot )}^R}} \lesssim_d \mathop {\sup }\limits_{Q \subseteq \Omega } {\left\| {{\A_{\alpha ,Q}}} \right\|_{L^{p( \cdot )}(W,\Omega ) \to {L^{q( \cdot )}}(W,\Omega )}}.$$
Hence we have finished this proof process. 
\end{proof}

Through the above characterization, we can get the variable exponents generalization of the classical results as follows.

	\begin{cor}\label{cor-3.8.1}
Let $\pp, \qq \in \P(\Omega)$ with $\frac{1}{{p( \cdot )}} - \frac{1}{{q( \cdot )}} = \frac{\alpha }{n} \in \left( {0,1} \right)$. If ${{\calM}_{\alpha,W}}$ is bounded from ${L^{p( \cdot )}}(\Omega )$ to ${L^{q( \cdot )}}(\Omega )$, then $W \in {\bbA_{p( \cdot ),q( \cdot )}}$.
	\end{cor}

\begin{proof}
	For any $Q \ni x$, we can obtain
	$$\left| {{\chi _Q}(x){{\left| Q \right|}^{\frac{\alpha }{n} - 1}}\int_Q {W(x)\vf(y)dy} } \right| \le {\chi _Q}(x){\left| Q \right|^{\frac{\alpha }{n} - 1}}\int_Q {\left| {W(x)\vf(y)} \right|dy}  \le {\calM_{\alpha ,W}}(W\vf)(x).$$
	By Theorem \ref{cha.Apq}, we can complete the proof by
	$$\mathop {\sup }\limits_{Q \subseteq \Omega } {\left\| {\left| {{\A_{\alpha ,Q}}(\vf)} \right|} \right\|_{{L^{q( \cdot )}}(W,\Omega )}} \le {\left\| {{\calM_{\alpha ,W}}(W\vf)} \right\|_{{L^{q( \cdot )}}(\Omega )}} \lesssim {\left\| \vf \right\|_{{L^{p( \cdot )}}(W,\Omega )}}.$$
And therefore we complete this Corollary. 
\end{proof}	
 Before presenting the final result, we still need to introduce the following lemma.
\begin{lem}\label{triangleinequality}
	Let $\pp, \qq \in \P(\Omega)$ with $\frac{1}{{p( \cdot )}} - \frac{1}{{q( \cdot )}} = \frac{\alpha }{n} \in \left( {0,1} \right)$. If $\omega_1,\ldots,\omega_l \in A_{p( \cdot ),q( \cdot )}$, then $\sum\limits_{i = 1}^l {{\omega _i}} \in A_{p( \cdot ),q( \cdot )}$.
\end{lem}
\begin{proof}[Proof:]
	For every cube $Q$,
	\begin{align*}
		&{\left| Q \right|^{\frac{1}{{p( \cdot )}} - \frac{1}{{q( \cdot )}} - 1}}{\left\| {{\chi _Q}{{\left( {\sum\limits_{i = 1}^l {{\omega _i}} } \right)}^{ - 1}}} \right\|_{{L^{p'( \cdot )}}(\Omega )}}{\left\| {{\chi _Q}\left( {\sum\limits_{i = 1}^l {{\omega _i}} } \right)} \right\|_{L_{}^{q( \cdot )}(\Omega )}}\\
		\le& \sum\limits_{i = 1}^l {{{\left| Q \right|}^{\frac{1}{{p( \cdot )}} - \frac{1}{{q( \cdot )}} - 1}}{{\left\| {{\chi _Q}{{\left( {\sum\limits_{i = 1}^l {{\omega _i}} } \right)}^{ - 1}}} \right\|}_{{L^{p'( \cdot )}}(\Omega )}}{{\left\| {{\chi _Q}{\omega _i}} \right\|}_{L_{}^{q( \cdot )}(\Omega )}}}\\
		\le& \sum\limits_{i = 1}^l {{{\left| Q \right|}^{\frac{1}{{p( \cdot )}} - \frac{1}{{q( \cdot )}} - 1}}{{\left\| {{\chi _Q}{\omega _i}^{ - 1}} \right\|}_{{L^{p'( \cdot )}}(\Omega )}}{{\left\| {{\chi _Q}{\omega _i}} \right\|}_{L_{}^{q( \cdot )}(\Omega )}}}\\
		\le& \sum\limits_{i = 1}^l {{{\left[ {{\omega _i}} \right]}_{{A_{p( \cdot ),q( \cdot )}}}}}.
	\end{align*}
	Thus, we have 
	$${\left[ {\sum\limits_{i = 1}^l {{\omega _i}} } \right]_{{A_{p( \cdot ),q( \cdot )}}}} \le \sum\limits_{i = 1}^l {{{\left[ {{\omega _i}} \right]}_{{A_{p( \cdot ),q( \cdot )}}}}}.$$
\end{proof}

	\begin{cor} \label{cor-3.8.2}
Let $\pp, \qq \in \P(\Omega)$ with $\frac{1}{{p( \cdot )}} - \frac{1}{{q( \cdot )}} = \frac{\alpha }{n} \in \left( {0,1} \right)$, and $\ve \in \R^d$ is a unit vector, then $W \in {\bbA_{p( \cdot ),q( \cdot )}}$ implies $\left| {W\ve} \right|, \left\| W \right\| \in A_{p( \cdot ),q( \cdot )}$.
		Moreover,
		$${\left[ {\left\| W \right\|} \right]_{{A_{p( \cdot ),q( \cdot )}}}} \lesssim_d {\left[ W \right]_{\bbA_{p( \cdot ),q( \cdot )}^{}}}.$$
	\end{cor}

\begin{proof}
	According to Theorem \ref{cha.Apq}, the operators \({\A_{\alpha ,Q}}\) are uniformly bounded from \({L^{p( \cdot )}}(W,\Omega )\) to \({L^{q( \cdot )}}(W,\Omega )\) for all cubes \(Q\) within \(\Omega\).

Using \(\vf(y) = h\ve\), we find that \({\A_{\alpha ,Q}}\) are uniformly bounded from \({L^{p( \cdot )}}(\left| {W\ve} \right|, \Omega)\) to \({L^{q( \cdot )}}(\left| {W\ve} \right|, \Omega)\) for all cubes in \(\Omega\), indicating that \(\left| {W\ve} \right|\) is in class \(A_{p( \cdot ),q( \cdot )}\) according to Theorem \ref{Cha.Apq} for \(m=1\).

	By Lemmas \ref{opNorm:equiv} and \ref{triangleinequality}, 
	$${\left[ {\left\| W \right\|} \right]_{{A_{p( \cdot ),q( \cdot )}}}} \approx_d {\left[ {\sum\limits_{i = 1}^d {\left| {W{\ve_i}} \right|} } \right]_{{A_{p( \cdot ),q( \cdot )}}}} \le \sum\limits_{i = 1}^d {{{\left[ {\left| {W{\ve_i}} \right|} \right]}_{{A_{p( \cdot ),q( \cdot )}}}}}  \le d{\left[ W \right]_{{\bbA_{p( \cdot ),q( \cdot )}}}},$$ 
	where $\left\{ {{\ve_1}, \cdots ,{\ve_d}} \right\}$ is any orthonormal basis in $\R^d$ and $\ve=\ve_1$.
\end{proof}

\section{\bf The proof of Theorem \ref{FMVLP}}\label{longproof}

The necessity follows immediately from Theorem \ref{Cha.Apq} since $${\A_{\alpha ,B}}\left( {\left| {{f_1}} \right|, \cdots ,\left| {{f_m}} \right|} \right)(x) \lesssim {\M_\alpha }(\vec f)(x).$$

In the following we prove the sufficiency,
let us start by simplifying the details in four steps.




\begin{itemize}
\item[\textbullet] \textbf{Step 1}

For each \( t \in \{0, \tfrac{1}{3}\}^n \), we define
\[
\mathscr{D}_t = \left\{ 2^{-k} \left( (-1)^k t + j + [0,1)^n \right) : k \in \mathbb{Z},\ j \in \mathbb{Z}^n \right\}.
\]
This collection shares the same properties as \( \mathscr{D}_0 \). Using this definition, we define the dyadic multilinear fractional maximal operator by 
	$${\M_{\alpha}^{\D_t}}(\vec f)(x) = \mathop {\sup }\limits_{Q \in {\D_t}} {\left| Q \right|^{\frac{\alpha }{n}}}\left( {\prod\limits_{i = 1}^m {{{\left\langle {{|f_i|}} \right\rangle }_Q}} } \right){\chi _Q}(x),$$
	By \cite[(4.1)]{Sehba2018}, 
	$${\M_\alpha}(\vec f)(x){ \lesssim_{n,m,\alpha}}\sum \limits_{t \in {{\{ 0,1/3\} }^n}} {{\M_{\alpha}^{\D_t}}(\vec f)(x)}.$$
	
	Then, it is sufficient to prove for ${\M_{\alpha}^{\D_t}}$. Further, we only need to prove for ${\M_{\alpha }^d}={\M_{\alpha }^{\D_0}}$, where the difference of ${\D_0}$ and ${\D_t}$ only depends on $n$.
\item[\textbullet] \textbf{Step 2}
To proceed, we derive the Fatou's Property, that is Lemma~\ref{lemma:fatou}, to make a approximation for $\M_\alpha^d$.

For any $f_i \in L^{p_i(\cdot)}(\omega_i)$, there exists ${\left\{ {{f_{i,k}}} \right\}_k} \subseteq L_c^\infty$, which increase pointwise to \(f_i\), and are such that
	$$\mathop {\lim }\limits_{k \to \infty } {\rm{ }}{\M_\alpha^d}({f_{1,k}}, \cdots ,{f_{m,k}})(x) = {\M_\alpha^d}(\vec f)(x).$$
The proof of the above is similar to \cite[Lemma~3.30]{Cruz2013}. Then it suffices to prove that for $\vec{f} \subseteq (L_c^\infty)^m$.
\item[\textbullet] \textbf{Step 3}
For \(\pp \in \P_0\) and a given weight \(v\), we define \(\Lp_v\) as the quasi-Banach function space with the norm
$${\left\| f \right\|_{L_v^{p( \cdot )}}}: = {\left\| {{v^{\frac{1}{{p( \cdot )}}}}f} \right\|_{L_{}^{p( \cdot )}}}.$$ The norm has many of the same basic properties as the $\Lp$ norm.

Let $u(\cdot)=\omega(\cdot)^{q(\cdot)}$ and
$\sigma_l(\cdot)=\omega_l(\cdot)^{-p_l'(\cdot)}$, $l=1,\dots,m$, then
\[   (\omega_{l}(x)\sigma_{l}(x))^{p_{l}(x)}
=(\omega_{l}(x)^{p_{l}^{\prime}(x)-1})^{-p_{l}(x)}
= \omega_{l}(x)^{-p_{l}^{\prime}(x)}
= \sigma_{l}(x).
\]
Therefore,
$${\left\| {{\M_\alpha^d}({f_1}{\sigma _1}, \ldots ,{f_m}{\sigma _m})} \right\|_{L_u^{q( \cdot )}}} = {\left\| {{\M_\alpha^d}({f_1}{\sigma _1}, \ldots ,{f_m}{\sigma _m})\omega } \right\|_{L_{}^{q( \cdot )}}},$$
and for $l=1,\dots,m$,
\begin{align*}
\|f_{l}\|_{L_{\sigma_{l}}^{p_{l}(\cdot)}}
= \|f_{l}\sigma_{l}\omega_{l}\|_{L^{p_{l}(\cdot)}}.  
\end{align*}     
Thereby, it will suffice to prove that 
\begin{align}\label{EQ-boundenessofM2}
\|{\M_\alpha^d}(f_{1}\sigma_{1},\ldots,f_{m}\sigma_{m})\|_{\Lq_{u}}
\lesssim \prod\limits_{i = 1}^m {{{\left\| {{f_i}} \right\|}_{L_{{\sigma _i}}^{{p_i}( \cdot )}}}},
\end{align}
since we can replace $f_{l}$ by $f_{l}/\sigma_{l}$, $l=1,\ldots,m$. 
\item[\textbullet] \textbf{Step 4}
Taking for granted that \(\|f_l\|_{L^{p_l(\cdot)}_{\sigma_{l}}}=1\), for each \(l=1, \ldots, m\), Lemma \ref{lemma:mod-norm} implies
	$$\int_{\rn} | {f_l}(x){|^{{p_l}(x)}}{\sigma _l}(x){\mkern 1mu} dx \le 1.$$

Therefore, it suffices to show that
\begin{align}\label{Xiu_1}
\| \M_\alpha^d(f_{1}\sigma_{1}, \dots, f_{m}\sigma_{m}) \|_{L_u^{q(\cdot)}} \lesssim 1,
\end{align}
which, by Lemma \ref{lemma:mod-norm}, is equivalent to proving
\begin{equation}
\int_{\mathbb{R}^n} \M_\alpha^d(f_{1}\sigma_{1}, \ldots, f_{m}\sigma_{m})^{q(x)} u(x)\, dx \lesssim 1,
\end{equation}
where the implicit constant is independent of \( f_l \) for \( l = 1, \ldots, m \).

\end{itemize}


We now proceed to prove sufficiency. Without loss of generality, it is enough to consider the case where \( m = 2 \).
Define the functions
\begin{align*}
h_{1}=f_{1}\chi_{\lbrace f_{1}>1\rbrace}, \;
 h_{2}=f_{1}\chi_{\lbrace f_{1}\leq 1 \rbrace}, 
h_{3}=f_{2}\chi_{\lbrace f_{2}>1\rbrace},\;
 h_{4}=f_{2}\chi_{\lbrace f_{2} \leq 1\rbrace},
\end{align*}
and for brevity define
\[     \rho(1)= 1, \quad \rho(2)= 1, \quad
\rho(3)= 2, \quad \rho(4)= 2.  \]
Then,
\begin{align*}
&\int_\subRn \M_\alpha^{d}\left( f_{1}\sigma_{1},
f_{2}\sigma_{2}\right)(x)^{q(x)}u(x)\,dx\\
\le&
\int_\subRn \M_\alpha^{d}\left( h_{1}\sigma_{1}, h_{3}\sigma_{2}\right)(x)
^{q(x)}u(x)\,dx
+\int_\subRn \M_\alpha^{d}\left( h_{1}\sigma_{1}, h_{4}\sigma_{2}\right)(x)
^{q(x)}u(x)\,dx\\
&+ 
\int_\subRn \M_\alpha^{d}\left( h_{2}\sigma_{1},
h_{3}\sigma_{2}\right)(x)^{q(x)}u(x)\,dx 
+ \int_\subRn \M_\alpha^{d}\left( h_{2}\sigma_{1}, h_{4}\sigma_{2}\right)(x)
^{q(x)}u(x)\,dx\\
=:& I_{1}+I_{2}+I_{3}+I_{4}.
\end{align*}

We give the following multilinear fractional-type Calder\'on-Zygmund decomposition.

Fix $a>2^{mn-\alpha}$ and for each $k \in \Z$, we define 
$${\Omega _k} = \left\{ {x \in \rn:{\M_{\alpha}^d}({f_1}{\sigma _1}, \cdots ,{f_m}{\sigma _m})(x) > {a^k}} \right\}.$$
Then, \({\Omega_k} = \bigcup_{j \in \mathbb{Z}} {Q_j^k}\), where \(\left\{Q_j^k\right\}_j\) constitutes a family of non-overlapping maximal dyadic cubes, characterized by the property that
$${a^k} < {\left| {Q_j^k} \right|^{\frac{\alpha }{n}}}\prod\limits_{i = 1}^m {{{\left\langle {{f_i}{\sigma _i}} \right\rangle }_{Q_j^k}}} \le {2^{mn - \alpha }}{a^k}< {a^{k + 1}}.$$
The similar argument can be refered to \cite[p.228]{Moen}, \cite[p.1245]{Lerner}, \cite[Proposition 7.11]{H2014}, \cite[Theorem 5.3.1]{Gra1}, and \cite{Sawyer1982}.

Furthermore, the sets \(E_j^k = \Q \setminus \Omega_{k+1}\) are each disjoint for any $k,j$. Additionally, $\left|Q_j^k\right| \approx \left|E_j^k\right|$, for any $x \in X$, it
is follow obviously that
$$
\mathscr{M}_\alpha^{d}\left(f_1 \sigma_1, \cdots, f_m \sigma_m\right)(x) \lesssim \sum_{k, j} \left|Q_j^k\right|^{\frac{\alpha}{n}} \prod_{i_1}^m \langle f_i \sigma_i\rangle_{Q_j^k} \chi_{E_j^k}(x) .
$$

According to Lemma~\ref{cor:Ainfty} and Lemma~\ref{lemma:Ainfty-prop}, the measures \(u ,\, \sigma_l \in A_\infty\), for \(l=1, 2\). Furthermore, there exists a constant \(\beta > 1\) such that 
\[  u(\Q) \le \beta u(E_j^k) \quad and \quad  \sigma_l(\Q) \le \beta \sigma_l(E_j^k). \]
Consequencely, we obtain $u(\Q) \approx u(E_j^k)$ and $\sigma_l(\Q) \approx \sigma_l(E_j^k)$.

\textbf{Estimate for $I_1$:}
\begin{align*}
I_1 
&\lesssim \sum_{k,j}\int_{E_{j}^{k}}\prod_{l=1,3}
\bigg(\int_{\Q}{h_{l}\sigma_{\rho(l)}\,dy}\bigg)^{q(x)}
|\Q|^{(\frac{\alpha }{n}-2)q(x)}u(x)\,dx. \\
\end{align*}   
Since $h_1 \geq 1$ or $h_1 = 0$, by Lemma \ref{Xiu_1},
\begin{equation}\label{Bound_function_geq1}
\int_{\Q}h_1(y)\sigma_{1}(y)\,dy
\leq \int_\subRn f_1(y)^{p_1(y)}\sigma_1(y)\,dy \leq 1. 
\end{equation}
The above estimate is also holds for $h_3$. 

By \eqref{delta}, we see that for almost all $x\in \Q$, such that $\delta(\Q)\leq q_{-}(\Q)\leq q(x)$, which means that,
\begin{align*}
I_{1} 
&\leq  \sum_{k,j}\int_{E_{j}^{k}}
\prod_{l=1,3}\bigg(\int_{\Q}h_{l}(y)\sigma_{\rho(l)}(y)\,dy\bigg)
^{\delta(\Q)}|\Q|^{(\frac{\alpha }{n}-2)q(x)}u(x)\,dx\\
& \leq \sum_{k,j}\int_{E_{j}^{k}}
\prod_{l=1,3}\bigg(\frac{1}{\sigma_{\rho(l)}(\Q)}
\int_{\Q}h_{l}(y)^{\frac{p_{l}(y)}{(p_{l})_{-}(\Q)}}\sigma_{\rho(l)}(y)\,dy
\bigg)^{\delta(\Q)} \\
& \qquad \qquad \times 
\sigma_{\rho(l)}(\Q)^{\delta(\Q)}|\Q|^{(\frac{\alpha }{n}-2)q(x)}u(x)\,dx.
\end{align*}

Firstly, for the second term of the above estimation, we assume that 
\begin{equation}\label{EQ-I1estimte}
\int_{E_{j}^{k}}\prod_{l=1,3}
|\Q|^{(\frac{\alpha }{n}-2)q(x)}\sigma_{\rho(l)}(\Q)^{\delta(\Q)}u(x)\,dx
\lesssim \left(\sigma_{1}(\Q)^{\frac{1}{(p_{1})_{-}(\Q)}}
\sigma_{2}(\Q)^{\frac{1}{(p_{2})_{-}(\Q)}}\right)^{\delta(\Q)}.
\end{equation}

Secondly, by applying Hölder's inequality using the measure \(\sigma_{\rho(l)} \, dx\) for \(l=1, 3\), we demonstrate that
\begin{align}
&\bigg(({\sigma_{\rho(l)}(\Q)})^{-1}
\int_{\Q}h_{l}(y)^{\frac{p_{l}(y)}{(p_{l})_{-}(\Q)}}
\sigma_{\rho(l)}(y)\,dy\bigg)^{\delta(\Q)} \notag\\
\le& \bigg(({\sigma_{\rho(l)}(\Q)})^{-1}
\int_{\Q}h_{l}(y)^{\frac{p_{l}(y)}{(p_{l})_{-}}}\sigma_{\rho(l)}(y)\,dy
\bigg)^{(p_{l})_{-}\frac{\delta(\Q)}{(p_{l})_{-}(\Q)}}
= \langle h_{l}^{\frac{p_{l}(\cdot)}{(p_{l})_{-}}}
\rangle_{\sigma_{\rho(l)},\Q}^{(p_{l})_{-}\frac{\delta(\Q)}{(p_{l})_{-}(\Q)}}.\label{EQ-estimateHölder}
\end{align}
Since 
\begin{equation*}
1=\frac{\eta(\Q)}{(p_{1})_{-}(\Q)}+\frac{\eta(\Q)}{(p_{2})_{-}(\Q)},
\end{equation*}
then by \eqref{EQ-estimateHölder} and Young's inequality, we have
\begin{align}
I_1 
\nonumber &\lesssim \sum_{k,j}\left(\prod_{l=1,3}
\langle h_{l}^{\frac{p_{l}(\cdot)}{(p_{l})_{-}}}
\rangle_{\sigma_{\rho(l)},Q}^{(p_{l})_{-}\frac{\eta(\Q)}{(p_{l})_{-}(\Q)}}
\sigma_{\rho(l)}(\Q)^{\frac{\eta(\Q)}{(p_{l})_{-}(\Q)}}\right)^{\frac{{\delta (\Q)}}{{\eta (\Q)}}}\\
\label{eqn:final-I1-est}     &\lesssim \sum_{k,j}\left(\sum_{l=1,3}
\langle h_{l}^{\frac{p_{l}(\cdot)}{(p_{l})_{-}}}
\rangle_{\sigma_{\rho(l)},\Q}^{(p_{l})_{-}}\sigma_{\rho(l)}( Q_{j}^{k})\right)^{\frac{{\delta (\Q)}}{{\eta (\Q)}}}\\
\nonumber  &\lesssim  \sum_{\theta=1,c}\left(\sum_{k,j}\sum_{l=1,3}
\langle h_{l}^{\frac{p_{l}(\cdot)}{(p_{l})_{-}}}
\rangle_{\sigma_{\rho(l)},\Q}^{(p_{l})_{-}}\sigma_{\rho(l)}(
E_{j}^{k})\right)^\theta.\\
\intertext{ By 
Lemma~\ref{lemma:wtd-max-bound}, since $(p_l)_->1$,}
\nonumber     &\leq \sum_{\theta=1,c}\left(\sum_{l=1,3} \int_{\subRn}
M_{\sigma_{\rho(l)}}^{d}(h_{l}^{\frac{p_{l}(\cdot)}{(p_{l})_{-}}})(x)
^{(p_{l})_{-}}\sigma_{\rho(l)}(x)\,dx\right)^\theta\\
\nonumber     & \lesssim \sum_{\theta=1,c}\left(\sum_{l=1,3} \int_{\subRn}
h_{l}(x)^{p_{l}(x)}\sigma_{\rho(l)}(x)\,dx\right)^\theta\\
\nonumber     & \lesssim 1,
\end{align}
where $c={c_{n,\alpha,p_1(\cdot),p_2(\cdot)}}\ge1$, and it doesn't depend on $\Q$.

It remains to justify \eqref{EQ-I1estimte}. To begin, we simplify \eqref{EQ-I1estimte} by rearranging terms.
\begin{align*}
& \int_{E_{j}^{k}}\prod_{l=1,3}|\Q|^{({\frac{\alpha }{n} - 2})q(x)}\sigma_{\rho(l)}(\Q)^{\delta(\Q)}u(x)\,dx \\
&\leq  \prod_{l=1,3}
\bigg(\frac{\sigma_{\rho(l)}(\Q)}
{\|\omega_{\rho(l)}^{-1}\chi_\Q\|_{p_{\rho(l)}^{\prime}(\cdot)}}
\bigg)^{\delta(\Q)}\\
&\times
\int_{\Q}\bigg(\prod_{l=1,3}
\|\omega_{\rho(l)}^{-1}\chi_\Q\|_{p_{\rho(l)}^{\prime}(\cdot)}^{\delta(\Q)-q(x)}\bigg)
\bigg(\prod_{l=1,3}
|\Q|^{({\frac{\alpha }{n} - 2})q(x)}\|\omega_{\rho(l)}^{-1}\chi_{\Q}\|_{p_{\rho(l)}^{\prime}(\cdot)}^{q(x)}
u(x)\bigg)\,dx.
\end{align*}      
By the $A_{\vec{p}(\cdot),q(\cdot)}$ condition,
\begin{equation*}
\big\||\Q|^{(\frac{\alpha }{n} - 2)}\prod_{l=1 }^{2}
\|\omega_{l}^{-1}\chi_{\Q}\|_{p_{l}^{\prime}(\cdot)}\omega\chi_{\Q}
\big\|_{q(\cdot)}
\lesssim 1,
\end{equation*}
which together with Lemma~\ref{lemma:mod-norm} gives
\begin{equation} \label{eqn:modular-Ap}
\int_{\Q}\prod_{l=1 }^{2}
\|\omega_{l}^{-1}\chi_{\Q}\|_{p_{l}^{\prime}(\cdot)}^{q(x)}
|\Q|^{(\frac{\alpha }{n} - 2)q(x)}u(x)\,dx \lesssim 1.
\end{equation}

To prove \(\eqref{EQ-I1estimte}\), it suffices to verify that for \(l = 1, 2\),
\begin{equation}\label{EQ-I11estimate}
\bigg(\frac{\sigma_{l}(\Q)}
{\|\omega_{l}^{-1}\chi_{\Q}\|_{p_{l}^{\prime}(\cdot)}}\bigg)^{\delta(\Q)}
\lesssim  \sigma_{l}(\Q)^{\frac{\delta(\Q)}{(p_{l})_{-}(\Q)}}
\end{equation}        
and
\begin{equation}\label{EQ-I12estimate}
\|\omega_{l}^{-1}\chi_{\Q}\|_{p_{l}^{\prime}(\cdot)}^{\delta(\Q)-q(x)}
\lesssim 1.
\end{equation}

Let us prove~\eqref{EQ-I11estimate} as follows. Suppose that
$\|\omega_{l}^{-1}\chi_{\Q}\|_{p_{l}^{\prime}(\cdot)} >1$. We see that Lemma~\ref{lemma:mod-norm} and $(p_l^{\prime})_{\pm}(\Q)=(p_{l})_{\mp}(\Q)^{\prime}$, which implies that
\begin{equation*}
\bigg(\frac{\sigma_{l}(\Q)}
{\| \omega_{l}^{-1}\chi_\Q\|_{p_{l}^{\prime}(\cdot)}}\bigg)^{\delta(\Q)}
\leq \bigg( \sigma_{l}(\Q)
^{\frac{(p_{l})_{-}(\Q)^{\prime}-1}{(p_{l})_{-}(\Q)^{\prime}}} 
\bigg)^{\delta(\Q)}
=\sigma_{l}(\Q)^{\frac{\delta(\Q)}{(p_{l})_{-}(\Q)}}.
\end{equation*}
On the other hand, if $\|\omega_{l}^{-1}\chi_{\Q}\|_{p_{l}^{\prime}(\cdot)}
\leq 1$, 
\begin{align*}
\frac{\sigma_l(\Q)}{\|\omega_l^{-1}\chi_\Q\|_{p_l'(\cdot)}}
\leq  \sigma_l(\Q)^{\frac{(p_l)_+(\Q)^{'}-1}{(p_l)_+(\Q)^{'}}}  
= &\sigma_l(\Q)^{\frac{1}{(p_l)_+(\Q)}}  \\
= &\sigma_l(\Q)^{\frac{1}{(p_l)_-(\Q)}}
\sigma_l(\Q)^{\frac{1}{(p_l)_+(\Q)}-\frac{1}{(p_l)_-(\Q)}}. 
\end{align*}
Again by Lemma~\ref{lemma:mod-norm} and Lemma~\ref{lemma:lemma3.3},
\begin{align*}
\sigma_l(\Q)^{\frac{1}{(p_l)_+(\Q)}-\frac{1}{(p_l)_-(\Q)}}
&\leq  \|\omega_l^{-\frac{1}{2}}\chi_\Q\|_{2p_l'(\cdot)}
^{[2(p_l')_{-}]\big(\frac{1}{(p_l)_+(\Q)}-\frac{1}{(p_l)_-(\Q)}\big)}\\
&=\|\omega_l^{-\frac{1}{2}}\chi_\Q\|_{2p_l'(\cdot)}
^{[2(p_l')_{-}]\big(1-\frac{1}{(p_l)_+(\Q)^{'}}-1+\frac{1}{(p_l)_-(\Q)^{'}}\big)} \\
&= \|\omega_l^{-\frac{1}{2}}\chi_\Q\|_{2p_l'(\cdot)}
^{[2(p_l')_{-}]\big(\frac{1}{(p_l^{\prime})_+(\Q)}-\frac{1}{(p_l^{\prime})_-(\Q)}\big)} \\
& \le \|\omega_l^{-\frac{1}{2}}\chi_\Q\|_{2p_l'(\cdot)}
^{c[(2p_l^{\prime})_-(\Q)-(2p_l^{\prime})_+(\Q)]}\\
& \lesssim 1.
\end{align*}
Hence, \eqref{EQ-I11estimate} is valid.

Last, we prove \(\eqref{EQ-I12estimate}\) by assuming that
$
\|\omega_{l}^{-1}\chi_{\Q}\|_{p_{l}^{\prime}(\cdot)}
\le 1,
$
otherwise, there is nothing to prove. By Lemma \ref{q-relation} and Lemma \ref{vweight4}, 
$$\left\| {{\omega _l}^{ - 1}{\chi _{Q_j^k}}} \right\|_{{{p_l}^\prime ( \cdot )}}^{\delta (Q) - q(x)} = \left\| {{\omega _l}^{ - \frac{1}{m}}{\chi _{Q_j^k}}} \right\|_{{m{p_l}^\prime ( \cdot )}}^{m(\delta (Q) - q(x))} \lesssim 1.$$
Then \eqref{EQ-I12estimate} is obvious.

This completes the estimate of $I_1$.  

\textbf{Estimate for $I_2$:}
To estimate the term \( I_2 \), we partition the set of cubes \( \Q \) into three disjoint sets: small cubes near the origin, large cubes near the origin, and cubes of any size far from the origin. Specifically, let \( \{ P_i \}_{i=1}^{2^n} \) be the \( 2^n \) dyadic cubes adjacent to the origin with \( |P_i| \geq 1 \). These cubes are sufficiently large so that for any dyadic cube \( Q \) equal to or adjacent to \( P_i \) in the same quadrant and with \( |Q| = |P_i| \), we have \( u(Q) \geq 1 \) and \( \sigma_l(Q) \geq 1 \) for \( l = 1, 2 \). The existence of such cubes follows from Lemma~\ref{lemma:Ainfty-prop} and Lemma~\ref{cor:Ainfty}. Let \( P = \bigcup_i P_i \). With this setup, we can partition the cubes \( \Q \) into three disjoint sets:

\begin{align*}
\mathscr{F} & =\lbrace (k,j): \Q \subseteq P_i \text{ for some } i\rbrace,\\
\mathscr{G}& =\lbrace (k,j): P_i \subseteq \Q \text{ for some } i \rbrace,\\
\mathscr{H} & =\lbrace (k,j): \Q \cap P_i = \emptyset \text{ for all } i
\rbrace. 
\end{align*}
Much as we did above for $I_1$, we conclude that
\begin{align*}
I_2
\lesssim&
\sum_{k,j}\int_{E_{j}^{k}}\prod_{l=1,4}
\langle h_{l}\sigma_{\rho(l)}\rangle_{\Q}^{q(x)}{\left| {Q_j^k} \right|^{\frac{\alpha }{n}\cdot q(x)}}u(x)\,dx =\sum_{(k,j)\in\mathscr{F}}
+\sum_{(k,j)\in\mathscr{G}}
+\sum_{(k,j)\in\mathscr{H}}
:= J_{1}+J_{2}+J_{3}.
\end{align*}
Let us proceed to estimate each \( J_i \) individually.
%

{\bf{Estimate for $J_{1}$:}}
We then invoke $h_{4}\leq 1$ and $q_+<\infty$ to deduce that
\begin{align*}
J_{1}
& = \sum_{(k,j)\in \mathscr{F}}\int_{E_{j}^{k}}{\left| {Q_j^k} \right|^{\frac{\alpha }{n}\cdot q(x)}}
\prod_{l=1,4}\langle h_{l}\sigma_{\rho(l)}\rangle_{\Q}^{q(x)}u(x)\,dx
\\
& \leq \sum_{(k,j)\in\mathscr{F}}\int_{E_{j}^{k}}{\left| {Q_j^k} \right|^{\frac{\alpha }{n}\cdot q(x)}}
\langle h_{1}\sigma_{1}\rangle_{\Q}^{q(x)}\langle
\sigma_{2}\rangle_{\Q}^{q(x)}u(x)\,dx \\
& = \sum_{(k,j)\in\mathscr{F}}\int_{E_{j}^{k}}
\bigg(\int_{\Q} h_1\sigma_1\,dy\bigg)^{q(x)}
\sigma_{2}(\Q)^{\delta(\Q)}\sigma_{2}(\Q)^{q(x)-\delta(\Q)}|\Q|^{(\frac{\alpha }{n} - 2)q(x)}u(x)\,dx \\
\intertext{by inequalities~\eqref{Bound_function_geq1} and~\eqref{EQ-estimateHölder},}
& \leq \sum_{(k,j)\in\mathscr{F}}\int_{E_{j}^{k}}
\bigg(\int_{\Q} h_1\sigma_1\,dy\bigg)^{\delta(\Q)}
\sigma_{2}(\Q)^{\delta(\Q)}\sigma_{2}(\Q)^{q(x)-\delta(\Q)}|\Q|^{(\frac{\alpha }{n} - 2)q(x)}u(x)\,dx \\
& = \sum_{(k,j)\in\mathscr{F}}\int_{E_{j}^{k}}
\langle h_{1}\rangle_{\sigma_{1},\Q}^{\delta(\Q)}
\sigma_{2}(\Q)^{q(x)-\delta(\Q)}\sigma_{1}(\Q)^{\delta(\Q)}\sigma_{2}(\Q)^{\delta(\Q)}
|\Q|^{(\frac{\alpha }{n} - 2)q(x)}u(x)\,dx \\
&\leq
\sum_{(k,j)\in\mathscr{F}}\big(\sigma_{2}(\Q)+1\big)^{q_{+}(\Q)-\delta(\Q)}
\langle h_{1}^{\frac{p_{1}(\cdot)}{(p_{1})_{-}}}\rangle_{\sigma_{1},\Q}
^{(p_{1})_{-}\frac{\delta(\Q)}{(p_{1})_{-}(\Q)}}\\
&   \qquad \qquad\times
\int_{E_{j}^{k}}|\Q|^{(\frac{\alpha }{n} - 2)q(x)}\sigma_{1}(\Q)^{\delta(\Q)}\sigma_{2}(\Q)^{\delta(\Q)}u(x)\,dx.
\\
\intertext{Setting $\delta_- = \mathop {\inf }\limits_Q \delta(Q)$,}
&\overset{\text{(\ref{EQ-I1estimte})}}{\lesssim}
\big(\sigma_{2}(P)+1\big)^{q_{+}-\delta_-}\sum_{(k,j)\in\mathscr{F}}
\langle h_{1}^{\frac{p_{1}(\cdot)}{(p_{1})_{-}}}
\rangle_{\sigma_{1},\Q}^{(p_{1})_{-}\frac{\delta(\Q)}{(p_{1})_{-}(\Q)}} 
\sigma_{1}(\Q)^{\frac{\delta(\Q)}{(p_{1})_{-}(\Q)}}\sigma_{2}(\Q)^{\frac{\delta(\Q)}{(p_{2})_{-}(\Q)}}.\\
\intertext{Using Lemma~\ref{lemma:wtd-max-bound} and applying Young's inequality, we obtain}
&\leq \big(\sigma_{2}(P)+1\big)^{q_{+}-\delta_{-}}\sum_{(k,j)\in
	\mathscr{F}}\left(
\langle h_{1}^{\frac{p_{1}(\cdot)}{(p_{1})_{-}}}\rangle_{\sigma_{1},\Q}^{(p_{1})_{-}}
\sigma_{1}(\Q)+\sigma_{2}(\Q)\right)^{\frac{{\delta (\Q)}}{{\eta (\Q)}}}  \\
&\lesssim    \big(\sigma_{2}(P)+1\big)^{q_{+}-\delta_{-}}\sum_{\theta=1,c}\left(\sum_{(k,j)\in\mathscr{F}}
\left(\langle
h_{1}^{\frac{p_{1}(\cdot)}{(p_{1})_{-}}}\rangle_{\sigma_{1},\Q}^{(p_{1})_{-}}
\sigma_{1}(E_{j}^{k})+\sigma_{2}(E_{j}^{k})\right)\right)^\theta\\
& \lesssim \sum_{\theta=1,c}\left(\sum_{(k,j)\in \mathscr{F}}\int_{E_{j}^{k}}
{M}_{\sigma_{1}}^{d}\big(f_{1}^{\frac{p_{1}(\cdot)}{(p_{1})_{-}}}\big)(x)
^{(p_{1})_{-}}\sigma_{1}(x)\,dx
+ \sum_{(k,j)\in \mathscr{F}} \sigma_{2}(E_{j}^{k})\right)^\theta\\
& \lesssim \sum_{\theta=1,c}\left(\int_{\subRn} f_{1}(x)^{p_{1}(x)}\sigma_{1}(x)\,dx
+ \sigma_{2}(P)\right)^\theta\\
& \lesssim 1,
\end{align*}
where $c={c_{n,\alpha,p_1(\cdot),p_2(\cdot)}}\ge1$, and it doesn't depend on $\Q$.

\textbf{Estimate for $J_{2}$:} 
%
%

First, we establish the following two facts.
\begin{align}\label{cd1w}
{|\Q|^{\frac{\alpha}{n}  - 2}}\int_{\Q}{h_1\sigma_{1}\,d x}\int_{\Q}{h_{4}\sigma_{2}\,d x} \le {c_0}.
\end{align}
and
\begin{equation}\label{P_infty_Bounded}
\sigma_{1}(\Q)^{q_{\infty}}\sigma_{2}(\Q)^{q_{\infty}}
|\Q|^{( {\frac{\alpha }{n} - 2} )q_{\infty}}u(E_{j}^{k})
\lesssim \sigma_{1}(\Q)^{\frac{q_{\infty}}{(p_{1})_{\infty}}}
\sigma_{2}(\Q)^{\frac{q_{\infty}}{(p_{2})_{\infty}}}.
\end{equation}

Now let us estimate to $J_2$.
 Using Lemmas~\ref{lemma:p-infty-px} and~\ref{lemma:infty-bound} with \eqref{cd1w} and \eqref{Bound_function_geq1}, there
exists $t>1$ such that
\begin{align*}
J_2 
& \approx
\sum_{(k,j)\in \mathscr{G}}\int_{E_{j}^{k}}
 \bigg({c_0}^{-1}\cdot{|\Q|^{\frac{\alpha}{n} - 2}}\int_{\Q}{h_{1}\sigma_{1}d x}
\int_{\Q}h_{4}\sigma_{2}d x\bigg)^{q(x)}u(x)\,d x\\
& \lesssim \sum_{(k,j)\in\mathscr{G}} \int_{E_{j}^{k}}
 \bigg({c_0}^{-1}\cdot{|\Q|^{\frac{\alpha}{n} - 2}}\int_{\Q}{h_{1}\sigma_{1}d x}
\int_{\Q}h_{4}\sigma_{2}d x\bigg)^{q_\infty}u(x)\,d x\\
& \qquad \qquad \qquad 
+ \sum_{(k,j)\in \mathscr{G}}\int_{E_{j}^{k}}
\frac{u(x)}{(e+\left| x \right|)^{ntq_{-}}}\,dx \\
&\lesssim \sum_{(k,j)\in \mathscr{G}}
\prod_{l=1,4}\langle h_{l}\rangle_{\sigma_{\rho(l),\Q}}^{q_{\infty}}
\sigma_{1}(\Q)^{q_{\infty}}
\sigma_{2}(\Q)^{q_{\infty}}\mu(\Q)^{\left( {\eta - 2} \right)q_{\infty}}u(E_{j}^{k})+1\\
& \lesssim \left(\sum_{(k,j)\in \mathscr{G}}
\prod_{l=1,4}\langle h_{l}\rangle_{\sigma_{\rho(l)},\Q}^{q_{\infty}}
\sigma_{\rho(l)}(\Q)^{\frac{q_{\infty}}{(p_{\rho(l)})_{\infty}}} + 1\right).
\end{align*}
The last inequality is obtained by \eqref{P_infty_Bounded}.

Secondly, by Lemma~\ref{Gu6} and again by
Lemma~\ref{lemma:p-infty-cond}, we conclude that
\begin{multline} \label{sigma1_Averange_Bound}
\frac{1}{\sigma_{1}(\Q)}\int_{\Q}{h_{1}(y)\sigma_{1}(y)\,d \mu}
\lesssim 
\sigma_{1}(\Q)^{-1}\|h_{1}\|_{L_{\sigma_{1}}^{p_{1}(\cdot)}}
\|\chi_{\Q}\|_{L_{\sigma_{1}}^{p_{1}^{\prime}(\cdot)}} \\
\leq
\sigma_{1}(\Q)^{-1}\|f\|_{L_{\sigma_{1}}^{p_{1}(\cdot)}} 
\|\omega_{1}^{-1}\chi_{\Q}\|_{\cpap}
\lesssim  \sigma_{1}(\Q)^{\frac{1}{(p_{1}^{\prime})_{\infty}}-1}
\leq \sigma_{1}(Q_0)^{-\frac{1}{(p_{1})_{\infty}}}
\lesssim 1 .
\end{multline}

Since
$\frac{1}{p_{\infty}}=\frac{1}{(p_{1})_{\infty}}+\frac{1}{(p_{2})_{\infty}},$
we derive ~\eqref{P_infty_Bounded} and Young's inequality to arrive at
\begin{align}
\nonumber  J_{2}
& \lesssim \sum_{(k,j)\in \mathscr{G}}
\prod_{l=1,4}\langle h_{l}\rangle_{\sigma_{\rho(l)},\Q}^{q_{\infty}}
\sigma_{\rho(l)}(\Q)^{\frac{q_{\infty}}{(p_{\rho(l)})_{\infty}}} + 1\\
\label{eqn:J2-final-est} & \lesssim \sum_{(k,j)\in \mathscr{G}}
\left(\langle h_{1}\rangle_{\sigma_{1},\Q}^{(p_{1})_{\infty}}
\sigma_{1}(\Q)\right)^{\frac{{{q_\infty }}}{{{p_\infty }}}}
+ \sum_{(k,j)\in \mathscr{G}}
\left(\langle h_{4}\rangle_{\sigma_{2},\Q}^{(p_{2})_{\infty}}\sigma_{2}(\Q)\right)^{\frac{{{q_\infty }}}{{{p_\infty }}}}
+1\\
\label{eqn:J2-final-est-2} & \lesssim \left(\sum_{(k,j)\in \mathscr{G}}
\langle c_1^{-1}h_{1}\rangle_{\sigma_{1},\Q}^{(p_{1})_{\infty}}
\sigma_{1}(E_{j}^{k})
+ \sum_{(k,j)\in \mathscr{G}}
\langle h_{4}\rangle_{\sigma_{2},\Q}^{(p_{2})_{\infty}}
\sigma_{2}(E_{j}^{k}) +1\right)^{\frac{{{q_\infty }}}{{{p_\infty }}}}, \\
\intertext{Invoking Lemmas~\ref{lemma:p-infty-px}
and~\ref{lemma:infty-bound}, there exists $t>1$ such that the items in parentheses above are controlled by}
\nonumber  & \quad \sum_{(k,j)\in \mathscr{G}}\int_{E_{j}^{k}}
\langle c_1^{-1}h_{1}\rangle_{\sigma_{1},\Q}^{p_{1}(x)}
\sigma_{1}(x)\,dx
+ \sum_{(k,j)\in \mathscr{G}}\int_{E_{j}^{k}}
\frac{\sigma_{1}(x)}{(e+|x|)^{tn(p_{1})_{-}}}\,dx\\ 
\nonumber    & \qquad \qquad \qquad 
+\sum_{(k,j)\in \mathscr{G}}
\int_{E_{j}^{k}}{M}_{\sigma_{2}}^{d}h_{4}(x)^{(p_{2})_{\infty}}
\sigma_{2}(x)\,dx +1\\
\nonumber  & \lesssim \sum_{(k,j)\in \mathscr{G}}\int_{E_{j}^{k}}
\langle
c_1^{-1}h_{1}^{\frac{p_{1}(\cdot)}{(p_{1})_{-}}}\rangle_{\sigma_{1},\Q}^{(p_{1})_{-}}
\sigma_{1}(x)\,dx
+ \int_{\subRn}{M}_{\sigma_{2}}^{d}h_{4}(x)^{(p_{2})_{\infty}}
\sigma_{2}(x)\,dx + 1;\\
\intertext{by Lemma~\ref{lemma:wtd-max-bound} applied twice,} 
\nonumber  & \lesssim \int_{\subRn}
{M}_{\sigma_{1}}^{d}(h_{1}^{\frac{p_{1}(\cdot)}{(p_{1})_{-}}})(x)^{(p_{1})_{-}}
\sigma_{1}(x)\,dx
+ \int_{\subRn}h_{4}(x)^{(p_{2})_{\infty}}\sigma_{2}(x)\,dx + 1\\
\nonumber &\lesssim \int_{\subRn} h_{1}(x)^{p_{1}(x)}\sigma_{1}(x)\,dx
+ \int_{\subRn}h_{4}(x)^{(p_{2})_{\infty}}\sigma_{2}(x)\,dx + 1\\    
\intertext{Gathering Lemmas~\ref{lemma:p-infty-px} and \ref{lemma:infty-bound}, we obtain } 
\nonumber & \le \int_{\subRn}h_{4}(x)^{p_{2}(x)}\sigma_{2}(x)\,dx
+\int_{\subRn}\frac{\sigma_{2}(x)}{(e+|x|)^{tn(p_{2})_{-}}}\,dx +1\\
\nonumber & \lesssim 1,
\end{align}
which completes the estimate for $J_2$. 

Now we remain to prove the first two facts.

We prove \eqref{cd1w} as follows. It follows from the
condition of $A_{\vec{p}(\cdot),q(\cdot)}$ and
Lemma \ref{Gu6} that
\begin{align*}
&{|\Q|^{\frac{\alpha}{n} - 2}}\int_{\Q}{h_{1}\sigma_{1}\,d x}\int_{\Q}{h_{4}\sigma_{2}\,d x} \\
 \le&  \|\omega\chi_{\Q}\|_{\qq}^{-1}
\prod_{l=1}^{2}
\|\omega_{l}^{-1}\chi_{\Q}\|_{p_{l}^{\prime}(\cdot)}^{-1}
\|h_{1}\|_{L_{\sigma_{1}}^{\pap}}\|h_{4}\|_{L_{\sigma_{2}}^{\pbp}}
\|\chi_{\Q}\|_{L_{\sigma_{1}}^{\cpap}}\|\chi_{\Q}\|_{L_{\sigma_{2}}^{\cpbp}}\\
\le &
\|\omega\chi_{\Q}\|_{\qq}^{-1}\prod_{l=1}^{2}\|\omega_{l}^{-1}\chi_{\Q}\|_{p_{l}^{\prime}(\cdot)}^{-1}\|\omega_{1}^{-1}\chi_{\Q}\|_{\cpap}
\|\omega_{2}^{-1}\chi_{\Q}\|_{\cpbp}\\
= &
\|\omega\chi_{\Q}\|_{\qq}^{-1}.
\end{align*}
Since $1 \leq u(Q_j^k)$, we can bound the above by $c_0$ using Lemma \ref{lemma:mod-norm}.

Finally, we prove \eqref{P_infty_Bounded}. Since \(\sigma_l(\Q) \geq 1\) and \(u(\Q) \geq 1\), by applying the definition of \(A_{\vec{p}(\cdot), q(\cdot)}\) and repeatedly using Lemma~\ref{lemma:p-infty-cond}, we arrive at
\begin{align*}  
\bigg[\sigma_{1}(\Q)\sigma_{2}(\Q)\bigg]^{q_{\infty}}
& \lesssim 
\bigg( \|\omega_{1}^{-\frac{1}{2}}\chi_{\Q}\|_{2\cpap}^{2(p_{1}^{\prime})_{\infty}}
\|\omega_{2}^{-\frac{1}{2}}\chi_{\Q}\|_{2\cpbp}^{2(p_{2}^{\prime})_{\infty}}\bigg)^{q_{\infty}} \\
& = \bigg(\|\omega_{1}^{-1}\chi_{\Q}\|_{\cpap}^{(p_{1}^{\prime})_{\infty}-1}
\|\omega_{2}^{-1}\chi_{\Q}\|_{\cpbp}^{(p_{2}^{\prime})_{\infty}-1}\bigg)^{q_{\infty}}
\bigg(\prod_{l=1}^{2} \|\omega_{l}^{-1}\chi_{\Q}\|_{p_{l}^{\prime}(\cdot)}
\bigg)^{q_{\infty}}\\
&\lesssim \bigg( \|\omega_{1}^{-1}\chi_{\Q}\|_{\cpap}^{(p_{1}^{\prime})_{\infty}-1}
\|\omega_{2}^{-1}\chi_{\Q}\|_{\cpbp}^{(p_{2}^{\prime})_{\infty}-1}\bigg)^{q_{\infty}}
\bigg(\frac{|\Q|^{2 - \frac{\alpha }{n}}}{\|\omega\chi_{\Q}\|_{\qq}}\bigg)^{q_{\infty}}\\
& \lesssim\bigg(
\sigma_{1}(\Q)^{\frac{(p_{1}^{\prime})_{\infty}-1}{(p_{1}^{\prime})_{\infty}}}
\sigma_{2}(\Q)^{\frac{(p_{2}^{\prime})_{\infty}-1}{(p_{2}^{\prime})_{\infty}}}
\bigg)^{q_{\infty}}\frac{|\Q|^{(2 - \frac{\alpha }{n})q_{\infty}}}{u(\Q)}\\
&\leq\sigma_{1}(\Q)^{\frac{q_{\infty}}{(p_{1})_{\infty}}}
\sigma_{2}(\Q)^{\frac{q_{\infty}}{(p_{2})_{\infty}}}
\frac{|\Q|^{(2 - \frac{\alpha }{n})q_{\infty}}}{u(E_{j}^{k})}.
\end{align*}
This proves \eqref{P_infty_Bounded}.

\textbf{Estimate for $J_{3}$:}

Considering $\Q$ is  such that $(k,j)\in \mathscr{H}$, $\Q$ does not
contain the origin.  Observe that
$\dist(\Q,0) \geq \ell(\Q)$, there exists a constant
$R>1$ depending only on $n$  such that 
\begin{equation}\label{H_cube_estimate}
\sup_{x\in \Q} |x| \leq R \inf_{x\in \Q} |x|. 
\end{equation}
Thus, in light of the continuity of $q(\cdot)$, there exists $x_+$ in the closure of $\Q$ such
that $q_{+}(\Q)=q(x_{+})$, which together with $\qq \in LH$, for all $x \in \Q$, apply ~\eqref{H_cube_estimate} concludes that
\begin{align} \label{eqn:p-J3-est}
0 \leq q_{+}(\Q)-q(x)
\leq  &|q(x_{+})-q(x)|+ |q(x)-q_{\infty}| \notag\\
\leq &\frac{C_{\infty}}{\log(e+|x_{+}|)}
+\frac{C_{\infty}}{\log(e+|x|)}
\lesssim \frac{1}{\log(e+|x|)}.
\end{align}
Follow the same way, in view of $l=1,2$, we deduce $p_l(\cdot)$ satisfies
\begin{equation} \label{eqn:p2-J3-est}
|(p_{l})_{-}(\Q)-p_{l}(x)|
\lesssim \frac{1}{\log(e+|x|)}.
\end{equation}

\smallskip

To estimate \( J_3 \), we partition \( \mathscr{H} \) into two subsets based on the size of the cubes \(\Q\) with respect to \(\sigma_2\),
\[ 
\mathscr{H}_{1}=\lbrace (k,j) \in \mathscr{H}:
\sigma_2(\Q)\leq 1 \rbrace, \quad
\mathscr{H}_{2}=\lbrace (k,j) \in \mathscr{H}: \sigma_2(\Q)>1
\rbrace.
\]
The sums over $\mathscr{H}_1$ and $\mathscr{H}_2$ are denoted as $J_{3}^1$ and $J_{3}^2$, respectively.

We then invoke ~\eqref{eqn:p-J3-est}, Lemmas~\ref{lemma:p-infty-px}
and~\ref{lemma:infty-bound} to deduce that
\begin{align*}
J_3^1&=\sum_{(k,j)\in \mathscr{H}_{1}} \int_{E_{j}^{k}}
\prod_{l=1,4}\langle h_{l}\sigma_{\rho(l)} \rangle_{\Q}^{q(x)}
{\left| {Q_j^k} \right|^{\frac{\alpha }{n} \cdot q(x)}}u(x)\,dx \\
&\lesssim \sum_{(k,j)\in \mathscr{H}_{1}} \int_{E_{j}^{k}}
{\left| {Q_j^k} \right|^{\frac{\alpha }{n} \cdot {q_ + }(Q_j^k)}}\prod_{l=1,4}\langle h_{l}\sigma_{\rho(l)}
\rangle_{\Q}^{q_{+}(\Q)} u(x)\,dx
+\sum_{(k,j)\in \mathscr{H}_{1}}\int_{E_{j}^{k}}\frac{u(x)}{(e+|x|)^{tnq_{-}}}\,dx\\
& \leq \sum_{(k,j)\in \mathscr{H}_{1}}
\int_{E_{j}^{k}}{\left| {Q_j^k} \right|^{\frac{\alpha }{n} \cdot {q_ + }(Q_j^k)}}\prod_{l=1,4}\langle h_{l}\sigma_{\rho(l)}
\rangle_{\Q}^{q_{+}(\Q)} u(x)\,dx + 1.
\intertext{Since \( h_1 \geq 1 \), \( h_4 \leq 1 \), and \( \sigma_2(\Q) \leq 1 \), we apply \eqref{Bound_function_geq1}, \eqref{EQ-I1estimte}, and Lemma~\ref{lemma:diening} to obtain}
&= \sum_{(k,j)\in \mathscr{H}_{1}} \int_{E_{j}^{k}}
\bigg(\int_{\Q}{h_{1}\sigma_{1}\,dy}\bigg)^{q_{+}(\Q)}
\bigg(\frac{1}{\sigma_{2}(\Q)}\int_{\Q}{h_{4}\sigma_{2}\,dy}\bigg)^{q_{+}(\Q)}\\
& \qquad \times
|\Q|^{(\frac{\alpha }{n} - 2)q_{+}(\Q)}\sigma_{2}(\Q)^{q_{+}(\Q)}u(x)\,dx +1; \\
&\lesssim \sum_{(k,j)\in \mathscr{H}_{1}} 
\langle h_{1} \rangle_{\sigma_{1},\Q}^{\delta(\Q)}
\langle h_{4}\rangle_{\sigma_{2},\Q}^{\delta(\Q)}
\int_{E_{j}^{k}}|\Q|^{(\frac{\alpha }{n} - 2)q(x)}
\sigma_{1}(\Q)^{\delta(\Q)}\sigma_{2}(\Q)^{\delta(\Q)} u(x)\,dx + 1\\
&\lesssim \sum_{(k,j)\in \mathscr{H}_{1}}
\prod_{l=1,4}\langle h_{l}\rangle_{\sigma_{\rho(l),\Q}}^{\delta(\Q)}
\sigma_{1}(\Q)^{\frac{\delta(\Q)}{(p_{1})_{-}(\Q)}}\sigma_{2}(\Q)^{\frac{\delta(\Q)}{(p_{2})_{-}(\Q)}}
+ 1\\
&\leq  \sum_{(k,j)\in \mathscr{H}_{1}}
\langle   h_{1}^{\frac{p_{1}(\cdot)}{(p_{1})_{-}(\Q)}} \rangle_{\sigma_{1},\Q}^{\delta(\Q)}
\sigma_{1}(\Q)^{\frac{\delta(\Q)}{(p_{1})_{-}(\Q)}}
\langle h_{4}\rangle_{\sigma_{2},\Q}^{\delta(\Q)}
\sigma_{2}(\Q)^{\frac{\delta(\Q)}{(p_{2})_{-}(\Q)}} + 1;\\
\intertext{Hence, we derive H\"older's inequality and Young's inequality to demonstrate}
& \leq \sum_{(k,j)\in \mathscr{H}_{1}}
\langle h_{1}^{\frac{p_{1}(\cdot)}{(p_{1})_{-}}}
\rangle_{\sigma_{1},\Q}^{(p_{1})_{-}\frac{\delta(\Q)}{(p_{1})_{-}(\Q)}}
\sigma_{1}(\Q)^{\frac{\delta(\Q)}{(p_{1})_{-}(\Q)}}
\langle h_{4}\rangle_{\sigma_{2},\Q}^{\delta(\Q)}
\sigma_{2}(\Q)^{\frac{\delta(\Q)}{(p_{2})_{-}(\Q)}} +1\\
& \lesssim \sum_{\theta=1,c}\left(\sum_{(k,j)\in \mathscr{H}_{1}}
\langle
h_{1}^{\frac{p_{1}(\cdot)}{(p_{1})_{-}}}\rangle_{\sigma_{1},\Q}^{(p_{1})_{-}}
\sigma_{1}(\Q)
+ \sum_{(k,j)\in \mathscr{H}_{1}}
\langle h_{4}\rangle_{\sigma_{2},\Q}^{(p_{2})_{-}(\Q)}
\sigma_{2}(\Q) \right)^{\theta}+ 1\\
& =: \sum_{\theta=1,c}\left(J_{3}^{11}+J_{3}^{12}\right)^{\theta}+1,
\end{align*}
where $c={c_{n,\alpha,p_1(\cdot),p_2(\cdot)}}\ge1$, and it doesn't depend on $\Q$.

Much as we did above, it is easy to show that the estimate for $J_{3}^{11}$ is similar to the estimate for $I_1$, we can indeed do as \eqref{eqn:final-I1-est}.

Hence, we only need to bound $J_{3}^{12}$.  
By Lemmas~\ref{lemma:p-infty-px} and ~\ref{lemma:infty-bound} (applied
twice) and by Lemma~\ref{lemma:wtd-max-bound}, we have
\begin{align*}
J_{3}^{12}
& \lesssim  \sum_{(k,j)\in \mathscr{H}_{1}} \int_{E_{j}^{k}}
\langle
h_{4}\rangle_{\sigma_{2},\Q}^{(p_{2})_{-}(\Q)}\sigma_{2}(x)\,dx \\
&\lesssim \sum_{(k,j)\in \mathscr{H}_{1}}
\int_{E_{j}^{k}}\langle
h_{4}\rangle_{\sigma_{2},\Q}^{(p_{2})_{\infty}}
\sigma_{2}(x)\,dx
+ \sum_{(k,j)\in \mathscr{H}_{1}}
\int_{E_{j}^{k}}\frac{\sigma_{2}(x)}{(e+|x|)^{nt(p_{2})_{-}}}\,dx \\
&\leq \int_{\subRn}
{M}_{\sigma_{2}}^{d}h_{4}(x)^{(p_{2})_{\infty}}\sigma_{2}(x)\,dx
+
\int_{\subRn}\frac{\sigma_{2}(x)}{(e+|x|)^{nt(p_{2})_{-}}}\,dx\\
&\lesssim 1+ \int_{\subRn}h_{4}(x)^{(p_{2})_{\infty}}\sigma_{2}(x)\,dx \\
& \lesssim 1+ \int_{\subRn}h_{4}(x)^{p_{2}(x)}\sigma_{2}(x)\,dx
+\int_{\subRn}\frac{\sigma_{2}(x)}{(e+|x|)^{nt(p_{2})_{-}}}\,dx\\
& \lesssim 1.
\end{align*}
Thus, the estimate for $J_3^1$ is accomplished.

To estimate \( J_3^2 \), by Lemma~\ref{Gu6}, the preceding inequality implies that
\begin{multline}\label{H2-first-estimate}
\int_{\Q}{h_{1}\sigma_{1}\,dy}
\lesssim \|h_{1}\|_{L_{\sigma_{1}}^{\pap}} 
\|\chi_{\Q}\|_{L_{\sigma_{1}}^{\cpap}}
\lesssim \|f_{1}\|_{L_{\sigma_{1}}^{\pap}}
\|\omega_{1}^{-1}\chi_{\Q}\|_{\cpap}
\leq c_0\|\omega_{1}^{-1}\chi_{\Q}\|_{\cpap}.
\end{multline}
Analogously, one can get
\begin{equation}\label{H2-second-estimate}
\int_{\Q}{h_{4}\sigma_{2}\,dy}
\leq c_0\|\omega_{2}^{-1}\chi_{\Q}\|_{\cpbp}.
\end{equation}

We next divide $\mathscr{H}_2$ into two parts $\mathscr{H}_{21}  \cup\mathscr{H}_{22}$, where
\[ \mathscr{H}_{21} = \{ (k,j)\in \mathscr{H}_2 : \sigma_1(\Q) \geq 1 \},
\quad
\mathscr{H}_{22} = \{ (k,j)\in \mathscr{H}_2 : \sigma_1(\Q)< 1 \}.  \]
The corresponding sums over $\mathscr{H}_{21}$ and $\mathscr{H}_{22}$ are denoted by $J_3^{21}$ and $J_3^{22}$ respectively.

Let us estimate the sum $J_3^{21}$. 
Given  \eqref{H2-first-estimate} and  \eqref{H2-second-estimate}, by
Lemma~\ref{lemma:p-infty-px}, we obtain
\begin{align*}
&\sum_{(k,j)\in \mathscr{H}_{21}} \int_{E_{j}^{k}}\prod_{l=1,4}
\langle h_{l}\sigma_{\rho(l)}\rangle_{\Q}^{q(x)}{\left| {Q_j^k} \right|^{\frac{\alpha }{n} \cdot q(x)}}u(x)\,dx \\
&\lesssim \sum_{(k,j)\in \mathscr{H}_{21}}
\int_{E_{j}^{k}} \prod_{l=1,4}
\bigg(c_0^{-1}\|\omega_{\rho(l)}^{-1}\chi_{\Q}\|_{p_{\rho(l)}^{\prime}(\cdot)}^{-1}
\int_{\Q}h_{l}\sigma_{\rho(l)}\,dy \bigg)^{q(x)}\prod_{l=1}^{2}
\bigg(\frac{\|\omega_{l}^{-1}\chi_{\Q}\|_{p_{l}^{\prime}(\cdot)}}{|\Q|^{1 - \frac{\alpha }{{2n}}}}\bigg)^{q(x)}u(x)\,dx\\
& \lesssim \sum_{(k,j)\in
	\mathscr{H}_{21}}\int_{E_{j}^{k}}
\prod_{l=1,4}\bigg(c_0^{-1}\|\omega_{\rho(l)}^{-1}\chi_{\Q}\|_{p_{\rho(l)}^{\prime}(\cdot)}^{-1}
\int_{\Q}h_{l}\sigma_{\rho(l)}\,dy\bigg)^{q_\infty} 
\prod_{l=1}^{2}\bigg(\frac{\|\omega_{l}^{-1}\chi_{\Q}\|_{p_{l}^{\prime}(\cdot)}}
{|\Q|^{1 - \frac{\alpha }{{2n}}}}\bigg)^{q(x)}u(x)\,dx \\
&  + 
\sum_{(k,j)\in \mathscr{H}_{21}} \int_{E_{j}^{k}} \prod_{l=1}^{2}
\bigg(\frac{\|\omega_{l}^{-1}\chi_{\Q}\|_{p_{l}^{\prime}(\cdot)}}{|\Q|^{1 - \frac{\alpha }{{2n}}}}\bigg)^{q(x)}
\frac{u(x)}{(e+|x|)^{tnq_{-}}}\,dx \\
&=: J_3^{211}+ J_3^{212}.
\end{align*}

We first estimate $J_3^{212}$. Since \( \sigma_2(E_j^k) \gtrsim \sigma_{2}(\Q) \geq 1 \), applying \eqref{H_cube_estimate}, \eqref{eqn:modular-Ap}, and Lemma~\ref{lemma:infty-bound}, we conclude that
\begin{align*}
	J_3^{212}
& \leq \sum_{(k,j)\in \mathscr{H}_{21}}\sup_{x\in
	\Q}(e+|x|)^{-ntq_{-}}
\int_{\Q}{\prod_{l=1}^{2}\|\omega_{l}^{-1}\chi_{\Q}\|_{p_{l}^{\prime}(\cdot)}^{q(x)}
	|\Q|^{(\frac{\alpha }{n} - 2)q(x)}u(x)\,dx}\\
& \lesssim \sum_{(k,j)\in \mathscr{H}_{21}}\inf_{x\in \Q}(e+|x|)^{-ntq_{-}}\sigma_{2}(E_j^k)\\
& \lesssim \int_{\subRn}\frac{\sigma_{2}(x)}{(e+|x|)^{ntq_{-}}}\,dx \\
& \lesssim 1.
\end{align*}

It remains to justify $J_3^{211}$.
Since $\sigma_l(\Q)\geq 1$ for $l=1,2$,  we invoke Lemma~\ref{lemma:p-infty-cond} to arrive at
\begin{equation} \label{eqn:p-infty-Hest}
\frac{\sigma_{l}(\Q)}
{\|\omega_{l}^{-1}\chi_{\Q}\|_{p_l'(\cdot)}}
\lesssim\frac{\sigma_{l}(\Q)}
{\sigma_{l}(\Q)^{\frac{1}{(p_{l}^{\prime})_{\infty}}}}
=  \sigma_{l}(\Q)^{\frac{1}{(p_{l})_{\infty}}}. 
\end{equation}
Hence, we use ~\eqref{eqn:modular-Ap} and Young's inequality to deduce that   
\begin{align*}
	J_3^{211}&\lesssim \sum_{(k,j)\in
	\mathscr{H}_{21}}\int_{E_{j}^{k}}{\prod_{l=1,4}
	\langle h_{l} \rangle_{\sigma_{\rho(l)},\Q}^{q_{\infty}}
	\sigma_{1}(\Q)^{\frac{q_{\infty}}{(p_{1})_{\infty}}}
	\sigma_{2}(\Q)^{\frac{q_{\infty}}{(p_{2})_{\infty}}}}
\prod_{l=1}^{2}\bigg(\frac{\|\omega_{l}^{-1}\chi_{\Q}
	\|_{p_{l}^{\prime}(\cdot)}}{|\Q|}\bigg)^{q(x)}u(x)\,dx \\
& \leq \sum_{(k,j)\in \mathscr{H}_{21}}
\prod_{l=1,4} \langle h_{l} \rangle_{\sigma_{\rho(l)},\Q}^{q_{\infty}}
\sigma_{1}(\Q)^{\frac{q_{\infty}}{(p_{1})_{\infty}}}
\sigma_{2}(\Q)^{\frac{q_{\infty}}{(p_{2})_{\infty}}}
\int_{\Q}{\prod_{l=1}^{2}\|\omega_{l}^{-1}\chi_{\Q} 
	\|_{p_{l}^{\prime}(\cdot)}^{q(x)}|\Q|^{(\frac{\alpha }{n} - 2)q(x)}u(x)\,dx}\\
&\lesssim \left(\sum_{(k,j)\in \mathscr{H}_{21}}
\langle h_{1}\rangle_{\sigma_{1},\Q}^{(p_{1})_{\infty}}\sigma_{1}(\Q)
+\sum_{(k,j)\in \mathscr{H}_{21}}\langle
h_{4}\rangle_{\sigma_{2},\Q}^{(p_{2})_{\infty}}
\sigma_{2}(\Q)\right)^{\frac{{{q_\infty }}}{{{p_\infty }}}}.\\
\end{align*}     

The estimate above is the same as for \( J_2 \). Since \( \sigma_1(\Q) \geq 1 \), we obtain \eqref{sigma1_Averange_Bound} and can use inequality \eqref{eqn:J2-final-est} to estimate it. This immediately implies that \( J_{3}^{21} \lesssim 1 \).

To control \( J_{3}^{22} \), we proceed by replacing the exponent \( q_\infty \) with \( \delta(\Q) \) and begin the estimate in the same way. For \( x \in \Q \), it follows from \eqref{eqn:p2-J3-est} that
\[ \bigg|\frac{1}{q(x)}-\frac{1}{\delta(\Q)}\bigg|
\leq \bigg|\frac{1}{p_1(x)}-\frac{1}{(p_1)_-(\Q)}\bigg|
+ \bigg|\frac{1}{p_2(x)}-\frac{1}{(p_2)_-(\Q)}\bigg| 
\lesssim \frac{1}{\log(e+|x|)}.  \]

Much as we did above for $J_3^{21}$, it is easy to deduce that
\begin{align*}
J_3^{22}=& \sum_{(k,j)\in \mathscr{H}_{22}} \int_{E_{j}^{k}}\prod_{l=1,4}
\langle h_{l}\sigma_{\rho(l)}\rangle_{\Q}^{q(x)}{\left| {Q_j^k} \right|^{\frac{\alpha }{n} \cdot q(x)}}u(x)\,dx \\
&  \lesssim \sum_{(k,j)\in
	\mathscr{H}_{22}}\int_{E_{j}^{k}}
\prod_{l=1,4}\bigg(c_0^{-1}\|\omega_{\rho(l)}^{-1}\chi_{\Q}\|_{p_{\rho(l)}^{\prime}(\cdot)}^{-1}
\int_{\Q}h_{l}\sigma_{\rho(l)}\,dy\bigg)^{\delta(\Q)} 
\prod_{l=1}^{2}\bigg(\frac{\|\omega_{l}^{-1}\chi_{\Q}\|_{p_{l}^{\prime}(\cdot)}}
{|\Q|^{1 - \frac{\alpha }{{2n}}}}\bigg)^{q(x)}u(x)\,dx \\
& + 
\sum_{(k,j)\in \mathscr{H}_{22}} \int_{E_{j}^{k}} \prod_{l=1}^{2}
\bigg(\frac{\|\omega_{l}^{-1}\chi_{\Q}\|_{p_{l}^{\prime}(\cdot)}}{|\Q|^{1 - \frac{\alpha }{{2n}}}}\bigg)^{q(x)}
\frac{u(x)}{(e+|x|)^{tnq_{-}}}\,dx \\
& :=J_3^{221}+ J_3^{222}.
\end{align*}

The estimate of $J_3^{222}$ is the same as $J_3^{212}$. To estimate \( J_3^{221} \), we apply \eqref{eqn:modular-Ap}–\eqref{EQ-I11estimate}, \eqref{eqn:J2-final-est-2}, and Young's inequality to obtain
\begin{align*}
	J_3^{221}&\lesssim \sum_{(k,j)\in
	\mathscr{H}_{22}}\int_{E_{j}^{k}}\left({\prod_{l=1,4}
	\langle h_{l} \rangle_{\sigma_{\rho(l)},\Q}^{\delta(\Q)}
	\sigma_{l}(\Q)^{\frac{\delta(\Q)}{(p_l)_-(\Q)}}}\right)
\prod_{l=1}^{2}\bigg(\frac{\|\omega_{l}^{-1}\chi_{\Q}
	\|_{p_{l}^{\prime}(\cdot)}}{|\Q|^{1 - \frac{\alpha }{{2n}}}}\bigg)^{q(x)}u(x)\,dx \\
& \leq \sum_{(k,j)\in \mathscr{H}_{22}}
\left(\prod_{l=1,4}\langle h_{l} \rangle_{\sigma_{\rho(l)},\Q}^{\delta(\Q)}
\sigma_{l}(\Q)^{\frac{\delta(\Q)}{(p_l)_-(\Q)}}\right)
\int_{\Q}{\prod_{l=1}^{2}\|\omega_{l}^{-1}\chi_{\Q} 
	\|_{p_{l}^{\prime}(\cdot)}^{q(x)}|\Q|^{(\frac{\alpha }{n} - 2)q(x)}u(x)\,dx}\\
&\lesssim \left(\sum_{(k,j)\in \mathscr{H}_{22}}
\langle h_{1}\rangle_{\sigma_{1},\Q}^{(p_{1})_-(\Q)}\sigma_{1}(\Q)
+\sum_{(k,j)\in \mathscr{H}_{22}}\langle
h_{4}\rangle_{\sigma_{2},\Q}^{(p_{2})_-(\Q)}
\sigma_{2}(\Q)\right)^{\frac{{\delta (Q_j^k)}}{{\eta (Q_j^k)}}}\\
&\lesssim \sum_{\theta=1,c}\left(\sum_{(k,j)\in \mathscr{H}_{22}}
\langle h_{1}\rangle_{\sigma_{1},\Q}^{(p_{1})_-(\Q)}\sigma_{1}(\Q)
+\sum_{(k,j)\in \mathscr{H}_{22}}\langle
h_{4}\rangle_{\sigma_{2},\Q}^{(p_{2})_-(\Q)}
\sigma_{2}(\Q)\right)^\theta.
\end{align*}     
where $c={c_{n,\alpha,p_1(\cdot),p_2(\cdot)}}\ge1$, and it doesn't depend on $\Q$.

To estimate the second term in parentheses, we proceed as in the final estimate of \( J_3^{12} \).

The first term is estimated similarly to \( J_3^{11} \) above. Since \( h_1 \geq 1 \) and by Hölder's inequality, we have

\[ \langle h_{1} \rangle_{\sigma_{1}, \Q}^{(p_{1})_-(\Q)} \leq \langle h_{1}^{\frac{p_{1}(\cdot)}{(p_{1})_-(\Q)}} \rangle_{\sigma_{1}, \Q}^{(p_{1})_-(\Q)} \leq \langle h_{1}^{\frac{p_{1}(\cdot)}{(p_{1})_-}} \rangle_{\sigma_{1}, \Q}^{(p_{1})_-}. \]

This yields \( J_{3}^2 \leq J_3^{21} + J_3^{22} \lesssim 1 \).
Therefore, \( J_3 \leq J_{3}^1 + J_{3}^2 \lesssim 1 \) and \( I_2 \leq J_1 + J_2 + J_3 \lesssim 1 \).
This completes the desired estimate for \( I_2 \).

\textbf{Estimate for $I_4$:}

The estimate for $I_4$ is similar to that of $I_2$. Let us  decompose $I_4$ into the same parts as above and only give key inequalities to estimate it.
By the multilinear fractional Calder\'on-Zygmund cubes related to $\M^d_{\alpha}(h_2\sigma_1,h_4\sigma_2)$.  
We may decompose and define the sets $\mathscr{F}$, $\mathscr{G}$
and $\mathscr{H}$. In this situation, we define the sums over these sets by $N_1$, $N_2$ and $N_3$.

\textbf{Estimate for $N_{1}$:} 

Similar to the estimate of $J_1$, observe that $h_2, h_4 \leq 1$, we deduce that 
\begin{align*}
N_{1}
&= \sum_{(k,j)\in \mathscr{F}}\int_{E_{j}^{k}}{\prod_{l=2,4}\langle h_{l}\sigma_{\rho(l)}\rangle_{\Q}^{q(x)}|\Q|^{\frac{\alpha }{n}q(x)}u(x)\,dx}\\
&\leq \sum_{(k,j)\in \mathscr{F}}\int_{E_{j}^{k}}{\prod_{l=1}^{2}\langle \sigma_{l}\rangle_{\Q}^{q(x)}|\Q|^{\frac{\alpha }{n}q(x)}u(x)\,dx}\\
&= \sum_{(k,j)\in \mathscr{F}}\int_{E_{j}^{k}}
\prod_{l=1}^{2}
\sigma_{1}(\Q)^{\delta(\Q)}\sigma_{l}(\Q)^{q(x)-\delta(\Q)}\sigma_{2}(\Q)^{\delta(\Q)}|\Q|^{(\frac{\alpha }{n} - 2)q(x)}u(x)\,dx \\
&\leq \sum_{(k,j)\in \mathscr{F}}
\prod_{l=1}^{2}\bigg(1+\sigma_{l}(\Q) \bigg)^{q_{+}(\Q)-\delta(\Q)}
\int_{E_{j}^{k}}\sigma_{1}(\Q)^{\delta(\Q)}\sigma_{2}(\Q)^{\delta(\Q)}|\Q|^{(\frac{\alpha }{n} - 2)q(x)}u(x)\,dx\\
&\lesssim \prod_{l=1}^{2}\bigg(1+\sigma_{l}(P)\bigg)^{q_{+}-\delta_{-}}
\sum_{(k,j)\in \mathscr{F}}
\sigma_{1}(\Q)^{\frac{\delta(\Q)}{(p_{1})_{-}(\Q)}}\sigma_{2}(\Q)^{\frac{\delta(\Q)}{(p_{2})_{-}(\Q)}}\\
&\lesssim \sum_{\theta=1,c}\left(\sum_{(k,j)\in \mathscr{F}}\sigma_{1}(\Q)
+ \sum_{(k,j)\in \mathscr{F}}\sigma_{2}(\Q)\right)^\theta\\
&\lesssim \sum_{\theta=1,c}\left(\sum_{(k,j)\in \mathscr{F}}\sigma_{1}(E_{j}^{k})
+\sum_{(k,j)\in \mathscr{F}}\sigma_{2}(E_{j}^{k})\right)^\theta\\
&\leq \sum_{\theta=1,c}\left(\sigma_{1}(P)+\sigma_{2}(P) \right)^\theta\\
& \lesssim 1,
\end{align*}
where $c={c_{n,\alpha,p_1(\cdot),p_2(\cdot)}}\ge1$, and it doesn't depend on $\Q$.

\textbf{Estimate for $N_{2}$:} 

We modify the argument for $J_2$. 
Gathering the
definition of $A_{\vec{p}(\cdot),q(\cdot)}$ and Lemma~\ref{Gu6}, we obtain
\begin{align*}
&\quad {\left| {Q_j^k} \right|^{\frac{\alpha }{n} - 2}}\int_{\Q}{h_{2}\sigma_{1}\,dy}\int_{\Q}{h_{4}\sigma_{2}\,dy} \\
&  \lesssim \|\omega \chi_{\Q}\|_{\qq}^{-1}
\prod_{l=1}^{2}
\|\omega_{l}^{-1}\chi_{\Q}\|_{p_{l}^{\prime}(\cdot)}^{-1}
\|h_{2}\|_{L_{\sigma_{1}}^{\pap}}\|h_{4}\|_{L_{\sigma_{2}}^{\pbp}}
\|\chi_{\Q}\|_{L_{\sigma_{1}}^{\cpap}}\|\chi_{\Q}\|_{L_{\sigma_{2}}^{\cpbp}}\\
&   =
\|\omega\chi_{\Q}\|_{\qq}^{-1}\prod_{l=1}^{2}\|\omega_{l}^{-1}\chi_{\Q}\|_{p_{l}^{\prime}(\cdot)}^{-1}\|h_{2}\|_{L_{\sigma_{1}}^{\pap}}\|h_{4}\|_{L_{\sigma_{2}}^{\pbp}}\|\omega_{1}^{-1}\chi_{\Q}\|_{\cpap}
\|\omega_{2}^{-1}\chi_{\Q}\|_{\cpbp};\\
\intertext{Since
	$u(Q)\geq u(P_i)\geq 1$, by Lemma~\ref{lemma:mod-norm}, the above}
& \lesssim\|\omega\chi_{\Q}\|_{\qq}^{-1} \\
& \leq c_0.
\end{align*}
Therefore, by Lemma~\ref{lemma:p-infty-px},
\begin{align*}
N_{2}
&= \sum_{(k,j)\in \mathscr{G}}\int_{E_{j}^{k}}
\prod_{l=2,4}\langle h_{l}\sigma_{\rho(l)}\rangle_{\Q}^{q(x)}{\left| {Q_j^k} \right|^{\frac{\alpha }{n}\cdot {q(x)}}}u(x)\,dx \\
& \lesssim \sum_{(k,j)\in \mathscr{G}}
\int_{E_{j}^{k}}\bigg({c_0^{-1}}{{\left| {Q_j^k} \right|^{\frac{\alpha }{n} - 2}}}
\int_{\Q}h_{2}\sigma_{1}\,dy \int_{\Q}h_{4}\sigma_{2}\,dy\bigg)^{q(x)}u(x)\,dx\\
&\lesssim \sum_{(k,j)\in \mathscr{G}}
\int_{E_{j}^{k}} \prod_{l=2,4}\langle
h_{l}\sigma_{\rho(l)}\rangle_{\Q}^{q_{\infty}}{\left| {Q_j^k} \right|^{\frac{\alpha }{n} \cdot {q_\infty }}}u(x)\,dx
+ \sum_{(k,j)\in \mathscr{G}}\int_{E_{j}^{k}}\frac{u(x)}{(e+|x|)^{tnq_{-}}}\,dx.
\end{align*}
By Lemma~\ref{lemma:infty-bound}, the second sum of above is bounded by 1, which together with ~\eqref{P_infty_Bounded} gives
\begin{align*}
&\quad\sum_{(k,j)\in \mathscr{G}}\int_{E_{j}^{k}}\prod_{l=2,4}
\langle h_{l}\sigma_{\rho(l)}\rangle_{\Q}^{q_{\infty}}{\left| {Q_j^k} \right|^{\frac{\alpha }{n} \cdot {q_\infty }}}u(x)\,dx \\
&  = \sum_{(k,j)\in
	\mathscr{G}}\int_{E_{j}^{k}}|\Q|^{(\frac{\alpha }{n} - 2)q_{\infty}}\prod_{l=2,4}
\langle h_{l}\rangle_{\sigma_{\rho(l)},\Q}^{q_{\infty}}
\sigma_{1}(\Q)^{q_{\infty}}\sigma_{2}(\Q)^{q_{\infty}}u(x)\,dx
\\
& \lesssim \sum_{(k,j)\in \mathscr{G}}
\langle h_{2}\rangle_{\sigma_{1},\Q}^{q_{\infty}}
\sigma_{1}(\Q)^{\frac{q_{\infty}}{(p_{1})_{\infty}}}
\langle h_{4}\rangle_{\sigma_{2},\Q}^{q_{\infty}}
\sigma_{2}(\Q)^{\frac{q_{\infty}}{(p_{2})_{\infty}}},\\
\intertext{Consequently, we conclude from Young's inequality with Lemmas~\ref{lemma:wtd-max-bound},~\ref{lemma:p-infty-px},
~\ref{lemma:infty-bound} that}
& \lesssim \left(\sum_{(k,j)\in \mathscr{G}}
\langle h_{2}\rangle_{\sigma_{1},\Q}^{(p_{1})_{\infty}}\sigma_{1}(\Q)
+  \sum_{(k,j)\in \mathscr{G}}
\langle h_{4}\rangle_{\sigma_{2},\Q}^{(p_{2})_{\infty}}\sigma_{2}(\Q)\right)^{\frac{{{q_\infty }}}{{{p_\infty }}}}\\
& \lesssim  \left(\sum_{(k,j)\in \mathscr{G}}
\langle
h_{2}\rangle_{\sigma_{1},\Q}^{(p_{1})_{\infty}}\sigma_{1}(E_{j}^{k})
+  \sum_{(k,j)\in \mathscr{G}}
\langle
h_{4}\rangle_{\sigma_{2},\Q}^{(p_{2})_{\infty}}\sigma_{2}(E_{j}^{k})\right)^{\frac{{{q_\infty }}}{{{p_\infty }}}}\\ 
& \le 
\left(\int_{\subRn}{M}_{\sigma_{1}}^{d}h_{2}(x)^{(p_{1})_{\infty}}\sigma_{1}(x)\,dx
+\int_{\subRn}{M}_{\sigma_{2}}^{d}h_{4}(x)^{(p_{2})_{\infty}}\sigma_{2}(x)\,dx\right)^{\frac{{{q_\infty }}}{{{p_\infty }}}}\\
& \lesssim
\left(\int_{\subRn}h_{2}(x)^{(p_{1})_{\infty}}\sigma_{1}(x)\,dx
+ \int_{\subRn}h_{4}(x)^{(p_{2})_{\infty}}\sigma_{2}(x)\,dx\right)^{\frac{{{q_\infty }}}{{{p_\infty }}}}\\
& \lesssim \left(\int_{\subRn}h_{2}(x)^{p_{1}(x)}\sigma_{1}(x)\,dx
+\int_{\subRn}h_{4}(x)^{p_{2}(x)}\sigma_{2}(x)\,dx+ \sum\limits_{l = 1}^2 {\int_{\subRn}\frac{\sigma_{l}(x)}{(e+|x|)^{tn(p_{l})_{-}}}\,dx}\right)^{\frac{{{q_\infty }}}{{{p_\infty }}}} \\
& \lesssim 1.
\end{align*}

\textbf{Estimate for $N_{3}$:} 
Estimating \( N_3 \) is similar to \( J_3 \) above. We proceed by splitting \( \mathscr{H} \) into \( \mathscr{H}_1 \) and \( \mathscr{H}_2 \).
To divide $\mathscr{H}_1$, we need to fix some notations. Define
\[ \mathscr{H}_{11} = \{ (k,j) \in \mathscr{H}_{1} : \sigma_1(\Q)\leq
1, \sigma_2(\Q) \leq 1 \} \]
and
\[ \mathscr{H}_{12} = \{ (k,j) \in \mathscr{H}_{1} : \sigma_1(\Q)>
1, \sigma_2(\Q) \leq 1\}. \]
The corresponding sums over $\mathscr{H}_{11}$ and $\mathscr{H}_{12}$ are defined respectively as $N_3^{11}$ and $N_3^{12}$.

The estimate for $N_3^{11}$ is similar to $J_{3}^1$. 
Collecting Lemmas~\ref{lemma:p-infty-px}
and~\ref{lemma:infty-bound}, since $h_2,\,h_4\leq 1$, we deduce that
\begin{align*}
	N_3^{11}=& \sum_{(k,j)\in \mathscr{H}_{11}} \int_{E_{j}^{k}}
\prod_{l=2,4}\langle h_{l}\sigma_{\rho(l)} \rangle_{\Q}^{q(x)}
{\left| {Q_j^k} \right|^{\frac{\alpha }{n} \cdot q(x)}}u(x)\,dx \\
  &\lesssim \sum_{(k,j)\in \mathscr{H}_{11}} \int_{E_{j}^{k}}
\prod_{l=2,4}\langle h_{l}\sigma_{\rho(l)}
\rangle_{\Q}^{q_{+}(\Q)} {\left| {Q_j^k} \right|^{\frac{\alpha }{n} \cdot {q_{+}(\Q)}}}u(x)\,dx
+\sum_{(k,j)\in \mathscr{H}_{11}}\int_{E_{j}^{k}}\frac{u(x)}{(e+|x|)^{tnq_{-}}}\,dx\\
 &\le \sum_{(k,j)\in \mathscr{H}_{11}}
\prod_{l=2,4}\langle h_{l}
\rangle_{\sigma_{\rho(l)},\Q}^{q_{+}(\Q)}  
\int_{E_{j}^{k}}|\Q|^{(\frac{\alpha }{n} - 2) q_{+}(\Q)} \prod_{l=2,4}\sigma_{\rho(l)}(\Q)^{q_{+}(\Q)} u(x)\,dx + 1.
\intertext{
	Firstly, since $h_2,\, h_4\leq 1$ and $\sigma_l(\Q)\leq 1$, $l=1,\,2$, we invoke \eqref{EQ-I1estimte} and Lemma~\ref{lemma:diening} to arrive at} 
&  \leq \sum_{(k,j)\in \mathscr{H}_{11}}
\prod_{l=2,4}\langle h_{l}
\rangle_{\sigma_{\rho(l)},\Q}^{\delta(\Q)}  
\int_{E_{j}^{k}}|\Q|^{(\frac{\alpha }{n} - 2) q_{+}(\Q)}
\prod_{l=2,4}\sigma_{\rho(l)}(\Q)^{\delta(\Q)} u(x)\,dx + 1 \\
&\lesssim \sum_{(k,j)\in \mathscr{H}_{11}}
\prod_{l=2,4}\langle h_{l}\rangle_{\sigma_{\rho(l),\Q}}^{\delta(\Q)}
\sigma_{1}(\Q)^{\frac{\delta(\Q)}{(p_{1})_{-}(\Q)}}\sigma_{2}(\Q)^{\frac{\delta(\Q)}{(p_{2})_{-}(\Q)}}
+ 1;\\
\intertext{by Young's inequality,}
& \lesssim \sum_{\theta=1,c}\left(\sum_{(k,j)\in \mathscr{H}_{11}}
\langle
h_{2}\rangle_{\sigma_{1},\Q}^{(p_{1})_{-}(\Q)}
\sigma_{1}(\Q)
+ \sum_{(k,j)\in \mathscr{H}_{11}}
\langle h_{4}\rangle_{\sigma_{2},\Q}^{(p_{2})_{-}(\Q)}
\sigma_{2}(\Q)\right)^\theta+ 1.
\end{align*}
where $c={c_{n,\alpha,p_1(\cdot),p_2(\cdot)}}\ge1$, and it doesn't depend on $\Q$. Both of the final terms are estimated as $N_3^{12}$ above.

\medskip
Secondly, we replace the $h_1$ with $h_2$, in view of \eqref{H2-first-estimate},~\eqref{H2-second-estimate} and Lemma~\ref{lemma:p-infty-px}, we conclude that
\begin{align*}
&\sum_{(k,j)\in \mathscr{H}_{12}} \int_{E_{j}^{k}}\prod_{l=2,4}
\langle h_{l}\sigma_{\rho(l)}\rangle_{\Q}^{q(x)}{\left| {Q_j^k} \right|^{\frac{\alpha }{n} \cdot q(x)}}u(x)\,dx \\
  \lesssim& \sum_{(k,j)\in
	\mathscr{H}_{12}}\int_{E_{j}^{k}}
\prod_{l=2,4}\bigg(c_0^{-1}\|\omega_{\rho(l)}^{-1}\chi_{\Q}\|_{p_{\rho(l)}^{\prime}(\cdot)}^{-1}
\int_{\Q}h_{l}\sigma_{\rho(l)}\,dy\bigg)^{\delta(\Q)}
\prod_{l=1}^{2}\bigg(\frac{\|\omega_{l}^{-1}\chi_{\Q}\|_{p_{l}^{\prime}(\cdot)}}
{|\Q|^{1 - \frac{\alpha }{{2n}}}}\bigg)^{q(x)}u(x)\,dx \\
  &+
\sum_{(k,j)\in \mathscr{H}_{12}} \int_{E_{j}^{k}} \prod_{l=1}^{2}
\bigg(\frac{\|\omega_{l}^{-1}\chi_{\Q}\|_{p_{l}^{\prime}(\cdot)}}{|\Q|^{1 - \frac{\alpha }{{2n}}}}\bigg)^{q(x)}
\frac{u(x)}{(e+|x|)^{tnq_{-}}}\,dx \\
=:&N_3^{121}+ N_3^{122}.
\end{align*}

The estimate for $N_3^{122}$ is identical to the estimate for $J_3^{212}$. Hence, we obtain
\begin{align*}
N_3^{121}&\lesssim \sum_{(k,j)\in
	\mathscr{H}_{12}}\int_{E_{j}^{k}}\left({\prod_{l=2,4}
	\langle h_{l} \rangle_{\sigma_{\rho(l)},\Q}^{\delta(\Q)}
	\sigma_{l}(\Q)^{\frac{\delta(\Q)}{(p_l)_-(\Q)}}}\right)
\prod_{J=1}^{2}\bigg(\frac{\|\omega_{J}^{-1}\chi_{\Q}
	\|_{p_{J}^{\prime}(\cdot)}}{|\Q|^{1 - \frac{\alpha }{{2n}}}}\bigg)^{q(x)}u(x)\,dx \\
& \leq \sum_{(k,j)\in \mathscr{H}_{12}}
\left(\prod_{l=2,4}\langle h_{l} \rangle_{\sigma_{\rho(l)},\Q}^{\delta(\Q)}
\sigma_{l}(\Q)^{\frac{\delta(\Q)}{(p_{l})_-(\Q)}}\right)
\int_{\Q}{\prod_{J=1}^{2}\|\omega_{J}^{-1}\chi_{\Q} 
	\|_{p_{J}^{\prime}(\cdot)}^{q(x)}|\Q|^{(\frac{\alpha }{n} - 2)q(x)}u(x)\,dx}\\
&\lesssim \sum_{\theta=1,c}\left(\sum_{(k,j)\in \mathscr{H}_{12}}
\langle h_{2}\rangle_{\sigma_{1},\Q}^{(p_{1})_-(Q_j^k)}\sigma_{1}(\Q)
+\sum_{(k,j)\in \mathscr{H}_{12}}\langle
h_{4}\rangle_{\sigma_{2},\Q}^{(p_{2})_-(Q_j^k)}
\sigma_{2}(Q_j^k)\right)^\theta.
\end{align*}     
where $c={c_{n,\alpha,p_1(\cdot),p_2(\cdot)}}\ge1$, and it doesn't depend on $\Q$.

Analogously, same as the estimate for $J_3^{12}$, one can get $N_3^{12} \lesssim 1$.

To calculate the total for $\mathscr{H}_{2}$, $N_{3}^{2}$, which is similar to that of $J_3$, we examine the sets $\mathscr{H}_{21}$ and $\mathscr{H}_{22}$.

The estimation of $\mathscr{H}_{21}$, denoted as $N_3^{21}$, is similar to our previous estimation of this set but with $h_1$ replaced by $h_2$. This substitution generates two terms analogous to $J_3^{211}$ and $J_3^{212}$, which we label as $N_3^{211}$ and $N_3^{212}$, respectively. The estimation of $N_3^{212}$ mirrors that of $J_3^{212}$, and $N_3^{211}$ is estimated similarly to $J_3^{211}$. Note that in the final step, the $h_2$ term can be estimated just like the $h_4$ term since both satisfy $h_2, h_4 \leq 1$. Therefore, we readily obtain $N_3^{21} \lesssim 1$.

To estimate the sum over $\mathscr{H}_{22}$, denoted as $N_3^{22}$, we proceed similarly but replace $h_1$ with $h_2$. This yields terms analogous to $J_3^{221}$ and $J_3^{222}$, which we denote as $N_3^{221}$ and $N_3^{222}$, respectively. The estimation of $N_3^{221}$ remains the same as before. For $N_3^{222}$, we use the same reasoning but replace the exponent $\delta(\Q)$ with $q_\infty$.    Therefore, the final line of the estimate can be controlled by 
$$ \sum_{\theta=1,c}\left(\sum_{(k,j)\in \mathscr{H}_{22}}
\langle h_{1}\rangle_{\sigma_{1},\Q}^{(p_{1})_\infty}\sigma_{1}(\Q)
+\sum_{(k,j)\in \mathscr{H}_{22}}\langle
h_{4}\rangle_{\sigma_{2},\Q}^{(p_{2})_{\infty}}
\sigma_{2}(\Q)\right)^\theta.$$
Since the subsequent steps are entirely similar to those in the estimation of \( J_2 \), we readily conclude that \( N_3^{22} \lesssim 1 \). It follows immediately that \( I_4 \leq N_1 + N_2 + N_3 \lesssim 1 \).

In conclusion, we have \( \sum\limits_{i=1}^4 I_i \lesssim 1 \), which is the desired result.

\vspace{1cm}
\noindent{\bf Acknowledgements } 
The first author would like to thank Prof. Mingquan Wei for his suggestions and discussion. The authors also would like to thank the editors and reviewers for careful reading and valuable comments, which lead to the improvement of this paper.

\medskip 

\noindent{\bf Data Availability} Our manuscript has no associated data.

\medskip 
\noindent{\bf\Large Declarations}
\medskip 

\noindent{\bf Conflict of interest} The authors state that there is no conflict of interest.

\end{document}